\documentclass{article}

\usepackage{graphics,color} 
\usepackage{epsfig,graphicx} 
\usepackage{float,subfigure}
\usepackage[leqno]{amsmath} 
\usepackage{amssymb}  

\newcommand{\AAA}{{\mathcal A}}

\newcommand{\AC}{{\mathcal A}^{\mathsf C}}

\newcommand{\II}{{\mathbb I}}
\newcommand{\JJ}{{\mathbb J}}

\newcommand{\DA}{\partial \mathcal{A}}

\newcommand{\DAM}{\left [\DA\right]_{\mathcal{-}}}

\newcommand{\ee}{\varepsilon}

\newcommand{\cl}{\mathsf{cl}}
\newcommand{\co}{\mathsf{co}}
\newcommand{\Int}{\mathsf{int}}

\newcommand{\sgn}{{\mathsf{sign}}}
\newcommand{\meas}{{\mathsf{meas}}}
\newcommand{\NN}{{\mathbb N}}
\newcommand{\Nnn}{{\mathcal N}}
\newcommand{\RR}{{\mathbb R}}

\newcommand{\UU}{{\mathcal U}}
\newcommand{\XX}{{\mathcal X}}

\newcommand{\ds}{\displaystyle}

\newcommand{\Rset}{{\mathbb R}}

\newcommand{\tg}{\tilde{g}}
\newcommand{\bbar}[1]{\bar{\bar{#1}}}
\DeclareMathOperator*{\esup}{{\mathrm{ess.sup}}}

\newtheorem{thm}{Theorem}[section] 
\newtheorem{pr}{Proposition}[section] 
\newtheorem{cor}{Corollary}[section] 
\newtheorem{lem}{Lemma}[section] 
\newtheorem{defn}{Definition}[section] 
\newtheorem{rem}{Remark}[section] 

\title{\textbf{Barriers in nonlinear control systems with mixed constraints}}

\author{Willem Esterhuizen\footnotemark[1] \and Jean L\'{e}vine\footnotemark[1]}

\date{7 August 2015}
\begin{document}

\maketitle
\renewcommand{\thefootnote}{\fnsymbol{footnote}}

\footnotetext[1]{CAS, Math\'{e}matiques et Syst\`{e}mes,
 MINES ParisTech, PSL Research University, 35, rue Saint-Honor\'{e}, 77300 Fontainebleau, France.
(Email: \texttt{willem.esterhuizen@mines-paristech.fr, jean.levine@mines-paristech.fr}).}

\renewcommand{\thefootnote}{\arabic{footnote}}

\begin{abstract}
In this paper, we propose an extension to mixed multidimensional constraints of the problem of state and input
constrained control introduced in \cite{DeDona_siam}, where the \emph{admissible set}, namely the subset of the state space where the state and input constraints can be satisfied \emph{for all times}, was studied, with focus on its boundary. The latter may be divided in two parts, one of them being called \emph{barrier}, a semipermeable surface. We extend this notion of \emph{barrier} to the mixed case and prove that it can be constructed via a minimum-like principle involving the Karush-K\"uhn-Tucker multipliers associated to the constraints and a generalised gradient condition at its endpoints.
\end{abstract}

\begin{paragraph}{Keywords}
mixed state and input constraints, barrier, admissible set, nonlinear systems.
\end{paragraph}

\pagestyle{myheadings}
\thispagestyle{plain}
\markboth{W. ESTERHUIZEN AND J. L\'{E}VINE}{BARRIERS IN NONLINEAR CONTROL SYSTEMS WITH MIXED CONSTRAINTS }

\pagestyle{myheadings}
\thispagestyle{plain}

\section{Introduction}

This paper is an extension to \emph{mixed constraints} of the paper \cite{DeDona_siam}, the latter paper being devoted to the study of the \emph{admissible set} for a nonlinear system with \emph{pure} state and input constraints, namely constraints described by functions that depend on the state only and on the input only. The admissible set consists of all the initial conditions for which there exists a control such that the constraints are satisfied for all times. Its boundary can be divided into two complementary parts, one of which is called the \emph{barrier}, proven to satisfy a minimum-like principle, therefore allowing its construction. The barrier enjoys the special property called \emph{semi-permeability}: if the state, initiating from the interior of the admissible set, crosses the barrier, then it is guaranteed that it will violate the constraints in the future. Moroever, if the state starts outside the admissible set, no admissible trajectory can cross the barrier in the direction of the interior of the admissible set.

In the current paper the above results are extended to the case where the constraints are \emph{mixed} (see \eqref{eq:state_const}) that is, they explicitly depend upon both the control and the state, without separation of these variables. Constraints of this type have been considered in the context of optimal control: for general theoretical results, the reader may refer to \cite{Hestenes,Clarke_DiPinho,Hartl_et_al} and, for applications where these constraints occur,  to \cite{NicotraNalGor_IFACE2014} in the context of tethered UAVs, or \cite{Pesch94apractical} in other aerospace applications.

Note that, as opposed to the latter references, no \textit{a priori} optimality notion in any sense is considered in this paper. More precisely, our approach may be applied to optimisation problems as a first step to prepare and simplify them by restricting the state space to the \emph{admissible set} where an optimal solution, if any, may be found.

Another important motivation to study mixed constraints is provided by flat systems \cite{JLbook,Sira}, submitted to constraints: if we express the state and control variables in terms of a flat output $y$, namely $x= \varphi(y,\dot{y},\ldots,y^{(\alpha)})$ and $u= \psi(y,\dot{y},\ldots, y^{(\alpha+1)})$, where $y^{(k)}$ denotes the $k$th order time derivative of $y$ for an arbitrary integer $k$, the constraint $\gamma(u)\leq 0$ is transformed into $\gamma(\psi(y,\dot{y},\ldots, y^{(\alpha+1)})) \triangleq \tilde{\gamma}(y,\dot{y},\ldots,y^{(\alpha)},v) \leq 0$ where $v=y^{(\alpha +1)}$ is the new control variable, the latter constraint described by $\tilde{\gamma}$ with respect to the transformed variables being naturally of a mixed nature.

It turns out that the concepts of \emph{barrier} and \emph{semi-permeability} carry over to the mixed constraint setting and that we can construct the barrier via a minimum principle, though containing significant modifications compared to the one of \cite{DeDona_siam}. 

The main contribution associated with this generalisation is threefold: 
\begin{itemize} 
	\item Since the control $u$ is only assumed measurable with respect to time, the evolution of the constraints along the integral curves of the system may be discontinuous and requires using tools from nonsmooth analysis \cite{Clarke,Clarke_et_al_springer}.
	\item Contrary to intuition, the equation satisfied by endpoints of the barrier is not $g_i(x,u) = 0$ for some $i$, $x$ and $u$ according to \eqref{eq:state_const}. We prove that in fact the endpoints satisfy $\tilde{g}(x) \triangleq \min_{u\in U} \max_{i} g_i(x,u) = 0$ with additional generalised gradient conditions (Proposition~\ref{ult-tan-1d-pr}).
	\item To prove the minimum principle associated with the barrier (Theorem~\ref{BarrierTheorem}), we use the same duality-like argument as in \cite{DeDona_siam}: the boundary of the constrained reachable set at some time $t$, issued from any point of the barrier, is tangent to the barrier, and the respective normals of both boundaries are opposed. However, the characterisation of the extremum trajectories whose endpoints lie in the boundary of the reachable set, which constitutes the main step to prove the maximum principle, in the spirit of \cite{Lee_Markus}, had to be generalised to the mixed constraint case (see Appendix \ref{Appendix:PMP}), assuming that the extremum control is piecewise continuous. This generalisation mainly consists in the construction of suitable \emph{needle perturbations} that satisfy the constraints (Section~\ref{AppendixSec:ElementaryPert}) to generate the celebrated \emph{perturbation cone} introduced by Pontryagin and coauthors \cite{PBGM}, a cone separated from the non-reachable part of the state space by a hyperplane whose normal is, at almost every instant of time, precisely the \emph{adjoint} vector.
\end{itemize}

The paper is organised as follows: in Section \ref{sec:ConsDynCon} the problem of characterising the admissible set in the mixed constraint case is presented along with the assumptions. Then we prove that this set is closed in Section \ref{sec:AdmSetTopol} and study its boundary in Section \ref{sec:BoundAdmSet} with an emphasis on the geometric properties of the barrier in Subsections \ref{subsec:GeometricDesc} and \ref{subsec:UltimateTangen}. Then Section \ref{sec:BarrierEquation} is devoted to the derivation of the minimum principle associated with the barrier and is followed by examples in Section \ref{sec:Examples}. Final remarks and conclusions are presented in Section \ref{sec:Conclusion} and two appendices on the compactness of solutions and the maximum principle in the mixed constraint case are given in Appendix \ref{Appendix:Compactness} and Appendix \ref{Appendix:PMP} respectively.

\section{Dynamical Control Systems with Mixed Constraints}\label{sec:ConsDynCon}

We consider the following constrained nonlinear system:
\begin{align}
\label{eq:state_space}
  & \dot{x} =  f(x,u), \\
  \label{eq:initial_condition}
  & x(t_0) = x_0, \\
  \label{eq:input_constraint}
 & u  \in \UU, \\
  \label{eq:state_const}
& g_i\big(x(t), u(t)\big)  \leq   0 \quad \mathrm{for~} \mathit{a.e.~} t \in [t_0, \infty) \quad i=1,...,p
\end{align}
where $x(t)\in \Rset^{n}$, $\Rset^{n}$ being endowed with the usual topology of the Euclidean norm.\footnote{We keep the same notation $\Vert\cdot\Vert$ for the Euclidean norm of $\Rset^{p}$ for every $p\geq 1$.} 
	
We denote by $U$ a given compact convex subset of $\Rset^{m}$, expressible as 
$$U \triangleq \{u\in\RR^m : \gamma_j(u)\leq 0, j = 1,\dots,r\}$$
with $r\geq m$, where the functions $\gamma_{j}$ are convex and of class $C^2$. Further assumptions on the functions $\{\gamma_j, j=1,\dots, r\}$ and $\{g_i, i=1,\dots,p\}$, associated to the constraints, are imposed in (A4)-(A5) (see below). The input function $u$ is assumed to belong to the set $\UU$ of Lebesgue measurable functions from $[t_0, \infty)$ to  $U$, i.e. $u$ is a measurable function such that $ u(t) \in U$ for almost all $t\in [t_0, \infty)$.

$x^u(t)$ or $x^{(u,x_0)}(t)$ denotes the solution of the differential equation~\eqref{eq:state_space} with input  $u\in \UU$  and initial condition~\eqref{eq:initial_condition}.	

Let us stress that the constraints \eqref{eq:state_const}, called \emph{mixed constraints} \cite{Hestenes,Clarke_DiPinho}, depend both on the state and the control. We denote by $g(x,u)$ the vector-valued function whose $i$-th component is $g_i(x,u)$. By $g(x,u)\prec 0$ (resp. $g(x,u)\preceq 0 $) we mean $g_i(x,u) < 0$ (resp. $g_i(x,u) \leq 0$) for all $i$. By $g(x,u)\circeq 0$, we mean $g_i(x,u) = 0$ for at least one $i$.

We define the following sets:
	\begin{gather} \ds 
	G \triangleq \bigcup_{u\in U} \{x\in\RR^n : g(x,u)\preceq 0\} 
	\\ \ds 
	G_0 \triangleq \{x \in G : \min_{u\in U} \max_{i\in\{1,...,p\}}g_i(x,u)=0 \}\label{def:G0} 
	\\ \ds 
	G_{-} \triangleq \bigcup_{u\in U} \{x\in\RR^n : g(x,u) \prec 0\} 
	\\ \ds 
	U(x) \triangleq \{u \in U : g(x,u) \preceq 0 \}\quad \forall x \in G.\label{def:U(x)}
	\end{gather}
	
Given a pair $(x,u)\in \RR^n \times U$, we denote by $\II(x,u)$ the set of indices, possibly empty, corresponding to the ``active'' mixed constraints, namely:
$$\II(x,u) = \{i_1,\dots,i_{s_{1}} \} \triangleq \{ i\in \{1,\ldots,p\} : g_{i}(x,u) = 0\}$$ 
and by $\JJ(u)$ the set of indices, possibly empty, corresponding to the ``active'' input constraints:
$$\JJ(u) = \{ j_1,\dots,j_{s_{2}}\} \triangleq \{j\in \{1,\ldots r\} : \gamma_{j}(u) = 0\}.$$ 
The integer $s_{1} \triangleq \#(\II(x,u)) \leq p$ (resp. $s_{2}\triangleq \#(\JJ(u)) \leq r$) is the number of elements of $\II(x,u)$  (resp. of $\JJ(u)$). Thus, $s_1 + s_2$ represents the number of ``active'' constraints, among the $p + r$ constraints, at $(x,u)$.

According to \cite{PBGM} a \emph{Lebesgue} point, also denoted by \emph{L-point} for notational convenience, for a given control $u\in \UU$ is a time $t\in [t_0,\infty)$ such that $u$ is continuous at $t$ in the sense that there exists a bounded (possibly empty) subset $I_{0} \subset [t_0,\infty)$, of zero Lebesgue measure, which does not contain $t$, such that $u(t)= \lim_{s\rightarrow t, s\not \in I_{0}} u(s)$.
Since $u$ is Lebesgue-measurable, by Lusin's theorem, the Lebesgue measure of the complement, in $[t_0,T]$, for all finite $T$, of the set of Lebesgue points is equal to 0.

Note that if $u_{1}\in \UU$ and $u_{2}\in \UU$, and if $\tau\geq t_{0}$ is given, the concatenated input $v$, defined by $v(t)= \left\{ \begin{array}{ll} u_{1}(t)&\mbox{\textrm if~} t\in [t_{0}, \tau[\\u_{2}(t)&\mbox{\textrm if~} t \geq \tau\end{array}\right.$ satisfies $v\in \UU$. The concatenation operator relative to $\tau$ is denoted by $\Join_{\tau}$, i.e. $v=u_{1}\Join_{\tau} u_{2}$.

We further assume:
\begin{description}
\item[(A1)] $f$ is an at least $C^{2}$ vector field of $\RR^{n}$ for every $u$ in an open subset $U_1$ of $\Rset^{m}$ containing $U$, whose dependence with respect to $u$ is also at least $C^{2}$.
\item[(A2)] There exists a constant $0 < C < +\infty$ such that the following inequality holds true:
 $$\sup_{u\in U}\vert x^{T}f(x,u) \vert \leq C(1+ \Vert x \Vert^{2} ), \quad \mbox{\textrm{for all}}~x$$
\item[(A3)] The set $f(x,U)$, called the \emph{vectogram} in \cite{Isaacs}, is convex for all $x\in \RR^{n}$.
\item[(A4)]\label{A4} $g$ is at least $C^{2}$ from $\RR^{n}\times U_1$ to $\RR^p$. 
Moreover, the (row) vectors
\begin{equation}\label{row-indep-cond}
\left\{ \frac{\partial g_{i}}{\partial u}(x,u), \frac{\partial \gamma_{j}}{\partial u}(u) : i\in \II(x,u), j\in \JJ(u)\right\}
\end{equation}
are linearly independent at every $(x,u) \in \RR^{n}\times U$ for which $\II(x,u)$ or $\JJ(u)$ is non empty.\footnote{Note that this implies that $s_1 + s_2 \leq m$, with $s_{1} = \#(\II(x,u))$ and $s_{2} = \#(\JJ(u))$} We say, in this case, that the point $x$ is \emph{regular} with respect to $u$ (see e.g. \cite{PBGM,Hestenes}).
\item[(A5)] For all $i=1,\ldots, p$, the mapping $u\mapsto g_{i}(x,u)$ is convex for all $x\in \RR^{n}$.
\end{description}
\bigskip

Given  $u\in \UU$, we will say that an integral curve $x^u$ of equation \eqref{eq:state_space} defined on $[t_0,T]$ is \emph{regular} if, and only if, at each L-point $t$ of $u$, $x^u(t)$ is regular in the afore mentioned sense w.r.t. $u(t)$, and, in the opposite case, namely if $t$ is a point of discontinuity of $u$, $x^u(t)$ is regular in the afore mentioned sense w.r.t. $u(t_-)$ and $u(t_+)$, with $u(t_-) \triangleq \lim_{\tau \nearrow t, t\notin I_0}u(\tau)$ and $u(t_+)\triangleq \lim_{\tau \searrow t, t\notin I_0}u(\tau)$, $I_0$ being a suitable zero-measure set of $\RR$.

Since system (\ref{eq:state_space}) is time-invariant, the initial time $t_0$ may be taken as 0. When clear from the context, ``$\forall t$" or ``for \emph{a.e} $t$'' will mean
``$\forall t \in [0, \infty)$" or ``for \emph{a.e.} $t\in [0, \infty)$''. Note that throughout this paper \emph{a.e.} is understood with respect to the Lebesgue measure.

\section{The Admissible Set: Topological Properties}\label{sec:AdmSetTopol} 
~

\begin{defn}[Admissible States]
\label{def:admiss_states}
We will say that a state-space point $\bar{x}$ is \emph{admissible} if there exists, at least, one input function $v\in \UU$, such
that~\eqref{eq:state_space}--\eqref{eq:state_const} are
satisfied for $x_0=\bar{x}$ and $u=v$:
\begin{equation}\label{eq:Admiss_states}
 \AAA \triangleq \{\bar{x} \in G: \exists u\in \UU,~ g\big(x^{(u,\bar{x})}(t), u(t)\big) \preceq 0, \mathrm{for~} \mathit{a.e.~} t\}.
\end{equation}
\end{defn}
According to the Markovian property of the system, any point of the integral curve,
$x^{(v,\bar{x})}(t')$, $t' \in [0, \infty)$, is also an admissible point.

The complement of $\AAA$ in $G$, namely $ \AC \triangleq G\setminus\AAA$, is thus given by:
\begin{equation}\label{eq:Compl_Admiss_states}
 \AC \triangleq \{\bar{x} \in G: \forall u\in \UU,~\exists i\in\{1,...,p\},~\exists \bar{t} < +\infty,~ \mathrm{L-point,~s.t.~} g_i\big(x^{(u,\bar{x})}(\bar{t}), u(\bar{t})\big) >  0\}.
\end{equation}

From now on, all set topologies will be defined relative to $G$. We assume that both  $\AAA$ and $\AC$ contain at least one element to discard the trivial cases  $\AAA = \emptyset$ and $\AC = \emptyset$. 

We use the notations $\Int(S)$ (resp. $\cl(S)$) (resp.$\co(S)$) for the interior (resp. the closure) (resp. the closed and convex hull) of a set $S$.

As in \cite{DeDona_siam}, we also consider the family of sets $\AAA_{T}$, called finite horizon admissible sets, defined for all finite $0\leq T<+\infty$ by
$$\AAA_T  \triangleq \{\bar{x} \in G: \exists u\in \UU,~ g\big(x^{(u,\bar{x})}(t), u(t)\big) \preceq  0, \mathrm{for~} \mathit{a.e.~} t \leq T\}$$
as well as its complement $\AC_{T}$ in $G$ is given by:
$$\AC_{T} \triangleq \{\bar{x} \in G: \forall u\in \UU,~\exists i\in\{1,...,p\},~\exists \bar{t} \leq T,~ \mathrm{L-point}
 ~ \mathrm{s.t.} ~ g_i\big(x^{(u,\bar{x})}(\bar{t}) , u(\bar{t})\big) >  0\}.$$
Clearly, since $\AAA \subset \AAA_T$ for all finite $T$, we have $\AAA_T\not = \emptyset$.

\begin{pr}
\label{closedness-prop}
Assume that (A1)--(A5) are valid. The set of finite horizon admissible states, $\AAA_{T}$, is closed for all finite $T$.
\end{pr}
\begin{proof}
The proof follows the same lines as Proposition 4.1 of \cite{DeDona_siam}, up to small changes. We sketch it for the sake of completeness.

Consider a sequence of initial states $\{ x_{k} \}_{k\in \NN}$ in $\AAA_T$ converging to $\bar{x}$ as $k$ tends to infinity. By definition of $\AAA_T$, for every $k\in \NN$, there exists $u_{k}\in \UU$ such that the corresponding integral curve $x^{(u_{k},x_{k})}$ satisfies $g(x^{(u_{k},x_{k})}(t), u_{k}(t)) \preceq 0$ for \emph{a.e.} $t\in [0,T]$. According to Lemma~\ref{compact-lem}, there exists a uniformly converging subsequence, still denoted by $x^{(u_{k},x_{k})}$, to the absolutely continuous integral curve $x^{(\bar{u},\bar{x})}$ for some $\bar{u}\in \UU$. Moreover, we have 
$g(x^{(\bar{u},\bar{x}))}(t), \bar{u}(t)) \preceq 0$ for almost every $t\in [0,T]$, hence $\bar{x}\in \AAA_T$, and the proposition is proven.
\end{proof}

\begin{cor}\label{close-cor}
Under the assumptions of Proposition~\ref{closedness-prop}, the set $\AAA$ is closed.
\end{cor}
\begin{proof}
See proof of Corollary 4.1 of \cite{DeDona_siam}.
\end{proof}

\section{Boundary of the Admissible Set}\label{sec:BoundAdmSet}
\subsection{A Characterisation of $\AAA$, its Complement and its Boundary}

Denoting by $\DA_{T}$ (resp. $\DA$) the boundary of $\AAA_{T}$ (resp. $\AAA$), we know from Proposition~\ref{closedness-prop} and Corollary~\ref{close-cor} that $\DA_{T}\subset \AAA_T$ (resp. $\DA \subset \AAA$). Following \cite{DeDona_siam}, we focus on the properties and
characterisation of these boundaries.

We  first prove the following result, where the notation $\esup_{t\in [0,\infty)} h(t)$, with $h : [0,\infty) \rightarrow \RR$  measurable, stands for the $L^{\infty}(0,\infty)$-norm of $h$.

\begin{pr} \label{min-sup-prop}
Assume that (A1)--(A5) hold.  We have the following equivalences:

\noindent (i) \, $\bar{x} \in \AAA$ is equivalent to
\begin{equation}\label{AAAopt}
\min_{u\in \UU} \esup_{t\in [0,\infty)} \max_{i=1,\ldots,p} g_{i}(x^{(u,\bar{x})}(t),u(t)) \leq 0
\end{equation}

\noindent (ii) \, $\bar{x} \in \AC$ is equivalent to
\begin{equation}\label{ACopt}
\min_{u\in \UU} \esup_{t\in [0,\infty)} \max_{i=1,\ldots,p} g_{i}(x^{(u,\bar{x})}(t),u(t)) >  0
\end{equation}

\noindent (iii) \, $\bar{x} \in \DA$ is equivalent to
\begin{equation}\label{DAopt}
\min_{u\in \UU} \esup_{t\in [0,\infty)} \max_{i=1,\ldots,p} g_{i}(x^{(u,\bar{x})}(t),u(t)) = 0.
\end{equation}
\end{pr}
\begin{proof}
We first prove (i). If $\bar{x} \in \AAA$, by definition, there exists $u\in \UU$ such that 
$g(x^{(u,\bar{x})}(t),u(t))\preceq 0$ 
for almost all $t\geq 0$, and thus such that 
$\esup_{t\in [0,\infty)}  \max_{i=1,\ldots,p} g_{i}(x^{(u,\bar{x})}(t),u(t))\leq 0$. 
We immediately get
\begin{equation}\label{infless}
\inf_{u\in \UU} \esup_{t\in [0,\infty)}  \max_{i=1,\ldots,p} g_{i}(x^{(u,\bar{x})}(t),u(t)) \leq 0.
\end{equation}
 Let us prove next that the infimum with respect to $u$ is achieved by some $\bar{u}\in \UU$ in order to get (\ref{AAAopt}). To this aim, let us consider a minimising sequence $u_{k} \in \UU$, $k\in \NN$, i.e. such that
\begin{equation}\label{minseq}
\lim_{k\rightarrow \infty}  \esup_{t\in [0,\infty)}  \max_{i=1,\ldots,p} g_{i}(x^{(u_{k},\bar{x})}(t),u_{k}(t)) = \inf_{u\in \UU} \esup_{t\in [0,\infty)}  \max_{i=1,\ldots,p} g_{i}(x^{(u,\bar{x})}(t),u(t)).
\end{equation}
According to Lemma~\ref{compact-lem} in  Appendix~\ref{Appendix:Compactness}, with $x_{k}=\bar{x}$ for every $k\in \NN$, one can extract a uniformly convergent subsequence on every compact interval $[0,T]$ with $T\geq 0$, still denoted by $x^{(u_{k},\bar{x})}$, whose limit is $x^{(\bar{u},\bar{x})}$ for some $\bar{u}\in \UU$. Moreover, one can build another subsequence, made of convex combinations of the $\{ g(x^{(u_{k},\bar{x})}, u_{k}) \}$, namely $\sum_{j=1}^{k} \alpha_{i,j}^{k} g_{i}(x^{(u_{j},\bar{x})}, u_{j})$, where the $\alpha_{i,j}^{k}$'s 
are all non negative real numbers such that $\sum_{j=1}^{k}\alpha_{i,j}^{k} = 1$ for all $i=1,\ldots, p$ and $k\geq 1$,  that pointwise converges to $g(x^{(\bar{u},\bar{x})}, \bar{u})$ a.e. $t \in [0,T]$ for all $T\geq 0$.

According to Egorov's theorem \cite{yosida}, the pointwise convergence implies that, for almost every $t\in [0,T]$, all $T\geq 0$ and $\ee >0$, there exists $k_{0}(t,T,\ee) \in \NN$ such that, for every $k\geq k_{0}(t,T,\ee)$, 
$$g_{i}(x^{(\bar{u},\bar{x})}(t), \bar{u}(t)) \leq \sum_{j=1}^{k} \alpha_{i,j}^{k} g_{i}(x^{(u_{j},\bar{x})}(t), u_{j}(t)) +\ee.$$ 
Taking the maximum with respect to $i \in \{ 1, \ldots, p \}$ and 
the essential supremum w.r.t.\ $t\in [0,\infty)$ on the right hand side, we get
$$
g_{i}(x^{(\bar{u},\bar{x})}(t), \bar{u}(t)) \leq  \sum_{j=1}^{k} \alpha_{i,j}^{k}  \esup_{s\in [0,\infty)} \max_{i=1,\ldots, p} g_{i}(x^{(u_{j},\bar{x})}(s), u_{j}(s)) + \ee 
\qquad \forall t \in [0,T]. 
$$

On the other hand, by the definition of the limit in (\ref{minseq}), for every $\ee >0$ 
there exists $k_{1}(\ee) \in \NN$ such that for all $j\geq k_{1}(\ee)$, we have
$$\esup_{t\in [0,\infty)}  \max_{i=1,\ldots, p}g_{i}(x^{(u_{j},\bar{x})}(t), u_{j}(t)) \leq  \inf_{u\in \UU} \esup_{t\in [0,\infty)}  \max_{i=1,\ldots, p}g_{i}(x^{(u,\bar{x})}(t), u(t)) + \ee$$
and thus
$$
g_{i}(x^{(\bar{u},\bar{x})}(t), \bar{u}(t)) \leq  \sum_{j=1}^{k} \alpha_{i,j}^{k}  \inf_{u\in \UU} \esup_{t\in [0,\infty)}  \max_{i=1,\ldots, p}g_{i}(x^{(u,\bar{x})}(t), u(t)) + 2\ee 
\qquad \forall t \in [0,T]. 
$$
Hence, using the fact that 
$\sum_{j=1}^{k}\alpha_{i,j}^{k} = 1$, 
for all $k\geq \max(k_{0}(t,T,\ee),k_{1}(\ee))$, we get
$$
g_{i}(x^{(\bar{u},\bar{x})}(t), \bar{u}(t)) \leq  \inf_{u\in \UU} \esup_{t\in [0,\infty)}  \max_{i=1,\ldots, p}g_{i}(x^{(u,\bar{x})}(t), u(t)) + 2\ee \quad \mbox{\textrm{a.e.}}~ t\in [0,T], \quad \forall i=1,\ldots,p.
$$
However, since the latter inequality is valid for any $t$ and
$T \geq 0$ and it does not depend on $k$ anymore, and since its right-hand side is independent of 
$i$,  
$t$ and $T$, we have that the inequality holds if we maximize the left-hand side with respect to  $i\in \{ 1,\ldots,p\}$ and take its essential supremum with respect to $t\in [0,\infty)$. 
Thus, using the definition of the infimum w.r.t.\ $u$, we obtain that,
for every $\ee > 0$
$$
\begin{aligned}
\esup_{t\in [0,\infty)} \max_{i=1,\ldots,p} g_{i}(x^{(\bar{u},\bar{x})}(t), \bar{u}(t)) 
&\leq  \inf_{u\in \UU} \esup_{t\in [0,\infty)}  \max_{i=1,\ldots, p}g_{i}(x^{(u,\bar{x})}(t), u(t)) + 2\ee 
\\&\leq  \esup_{t\in [0,\infty)}  \max_{i=1,\ldots, p}g_{i}(x^{(\bar{u},\bar{x})}(t), \bar{u}(t))+ 2\ee,
\end{aligned}
$$
or, using also~(\ref{infless}), that
$$ \esup_{t\in [0,\infty)}  \max_{i=1,\ldots, p}g_{i}(x^{(\bar{u},\bar{x})}(t), \bar{u}(t)) = \inf_{u\in \UU} \esup_{t\in [0,\infty)}  \max_{i=1,\ldots, p}g_{i}(x^{(u,\bar{x})}(t), u(t)) \leq 0,$$
which proves (\ref{AAAopt}).

Conversely, if (\ref{AAAopt}) holds, there exists an input $u\in \UU$ such that
$$\esup_{t\in [0,\infty)}  \max_{i=1,\ldots, p}g_{i}(x^{(u,\bar{x})}(t), u(t)) \leq 0,$$
which in turn implies that
$g(x^{(u,\bar{x})}(t), u(t))\preceq 0$ 
for almost all $t\geq 0$, or, in other words, $\bar{x} \in \AAA$, which achieves the proof of (i).

To prove (ii), we now assume that $\bar{x} \in \AC$ and prove (\ref{ACopt}). By definition of $\AC$, for all $u\in \UU$, we have 
$\esup_{t\in [0,\infty)}  \max_{i=1,\ldots, p}g_{i}(x^{(u,\bar{x})}(t)) > 0$ 
and thus
$$\inf_{u\in \UU}\esup_{t\in [0,\infty)}  \max_{i=1,\ldots, p}g_{i}(x^{(u,\bar{x})}(t), u(t)) \geq 0.$$
The same minimising sequence argument as in the proof of (i) shows that the minimum over $u\in \UU$ is achieved by some $\bar{u}\in \UU$ and that
$$\min_{u\in \UU}\esup_{t\in [0,\infty)}  \max_{i=1,\ldots, p}g_{i}(x^{(u,\bar{x})}(t), u(t)) \geq 0.$$
But the inequality has to be strict since, if 
$\ds \min_{u\in \UU}\esup_{t\in [0,\infty)}  \max_{i=1,\ldots, p}g_{i}(x^{(u,\bar{x})}(t), u(t)) = 0$, 
it would imply, according to (i), that $\bar{x}\in \AAA$ which contradicts the assumption. Therefore, we have proven  (\ref{ACopt}).

Conversely, if  (\ref{ACopt}) holds, it is immediately seen that $\bar{x}$ is such that 
\sloppy
 $\ds \esup_{t\in [0,\infty)}  \max_{i=1,\ldots, p}g_{i}(x^{(u,\bar{x})}(t), u(t)) > 0$ for all $u\in \UU$. 
\sloppy
The essential supremum with respect to $t$ must be reached at some $\bar{t}(u) < +\infty$ since $\bar{t}(u) = +\infty$ would imply that 
$\max_{i=1,\ldots, p}g_{i}(x^{(u,\bar{x})}(t), u(t)) \leq 0$ 
for almost all $t< +\infty$, and thus 
$\ds \esup_{t\in [0,\infty)}  \max_{i=1,\ldots, p}g_{i}(x^{(u,\bar{x})}(t), u(t)) \leq 0$. 
A fortiori, 
$\ds \min_{u\in \UU} \esup_{t\in [0,\infty)}  \max_{i=1,\ldots, p}g_{i}(x^{(u,\bar{x})}(t), u(t)) \leq 0$, 
which contradicts (\ref{ACopt}).
Thus, for all $u\in \UU$, there exists $\bar{t}(u) < +\infty$ such that 
$\ds \max_{i=1,\ldots, p} g_{i}(x^{(u,\bar{x})}(\bar{t}(u)), u(\bar{t}(u))) > 0$, 
and hence $\bar{x}\in \AC$, which proves (ii).

To prove (iii), since $\AAA$ is closed, $\bar{x} \in \DA$ is equivalent to $\bar{x}\in \AAA$ and $\bar{x}\in \cl(\AC)$, the closure of $\AC$, which, by (i) and (ii), is equivalent to (\ref{AAAopt}) and (\ref{ACopt}) (the latter with a ``$\geq$'' symbol as a consequence of $\bar{x}\in \cl(\AC)$), which in turn is equivalent to  
(\ref{DAopt}).
\end{proof}

\begin{rem}
The same formulas hold true for $\AAA_T$, $\AC_T$ and $\DA_T$ if one replaces the infinite time interval $[0,\infty)$ by $[0,T]$.
\end{rem}

\subsection{Geometric Description of the Barrier}\label{subsec:GeometricDesc}
As a consequence of \eqref{DAopt}, the boundary $\DA$ is made of points $\bar{x}$ such that there exists a $\bar{u}\in \UU$  for which at least one of the constraints is saturated for some L-point $\bar{t}$, i.e. $g(x^{(\bar{u},\bar{x})}(\bar{t}),\bar{u}(\bar{t}))\circeq 0$.  As in \cite{DeDona_siam}, let us define the set:
\[
\DAM = \DA\cap G_-
\]

\begin{defn}\label{def:barrier}
The set $\DAM$ is called the \emph{barrier} of the set $\AAA$ (see Corollary~\ref{bar-sem-cor} and \cite{DeDona_siam}).
\end{defn}

\begin{pr}\label{boundary:prop}
Assume (A1) to (A5) hold. $\DAM$ is made of points $\bar{x}\in G_-$ for which there exists $\bar{u}\in\UU$ and an integral curve $x^{(\bar{u},\bar{x})}$ entirely contained in $\DAM$ until it intersects $G_0$, i.e. at a point $z = x^{(\bar{u},\bar{x})}(\tilde{t})$, for some $\tilde{t}$, such that $\min_{u\in U} \max_{i=1,\ldots,p} g_{i}(z,u) = 0$.
\end{pr}

\begin{proof}
	Let $\bar{x}\in \DAM$, therefore satisfying (\ref{DAopt}). In particular, there exists $\bar{u}\in \UU$ and $\bar{t} > 0$ such that 
	$$\begin{aligned}
	\min_{u\in\UU}\esup_{t\in[0,\infty)}  \max_{i=1,\ldots, p}g_{i}(x^{u,\bar{x}}(t),u(t)) &= \esup_{t\in[0,\infty)}  \max_{i=1,\ldots, p}g_{i}(x^{\bar{u},\bar{x}}(t),\bar{u}(t)) \\
	&=  \max_{i=1,\ldots, p}g_{i}(x^{(\bar{u},\bar{x})}(\bar{t}),\bar{u}(\bar{t})) = 0
	\end{aligned}
	$$
	where $\bar{u}$ has been possibly modified on a 0-measure set to satisfy the right-hand side equality.
	Then, choose $\bar{t}$ as the first time for which $\max_{i=1,\ldots, p}g_{i}(x^{(\bar{u},\bar{x})}(\bar{t}),\bar{u}(\bar{t})) = 0$ and an arbitrary $t_0 \in [0,\bar{t}[$. Setting $\nu(t)=\bar{u}(t_0 +t)$, since $t_0 < \bar{t}$, the point $\xi = x^{(\bar{u},\bar{x})}(t_0)$ satisfies $\max_{i=1,\ldots, p}g_{i}(\xi, \nu(0)) = \max_{i=1,\ldots, p}g_{i}(x^{(\bar{u},\bar{x})}(t_{0}),\bar{u}(t_{0})) < 0$, i.e. $\xi\in G_{-}$, and by a standard dynamic programming argument (since $x^{(\nu,\xi)}(t)=x^{(\bar{u},\bar{x})}(t_0+t)$ for all $t\geq 0$),
	$\ds \min_{u\in \UU} \esup_{t\in [0,\infty)}  \max_{i=1,\ldots, p}g_{i}(x^{(u,\xi)}(t), u(t_0 +t)) = 0$. 
	It follows that $\xi\in \DAM$ 
	and, therefore, the arc of integral curve between 0 and $\bar{t}$ starting from $\bar{x}\in \DAM$ is entirely contained in $\DAM$. 
		
	We now prove that this integral curve intersects $G_0$. Since $\bar{x}\in \DAM$, there exists an open set ${\mathcal O} \subset \RR^n$ such that $\bar{x} +\ee h \in \AC$ for all $h \in {\mathcal O}$ and $\Vert h\Vert \leq H$, with $H$ arbitrarily small, and all $\ee$ sufficiently small. Therefore, there exists $t_{\ee,h}$ such that $ \max_{i=1,\ldots, p}g_{i}(x^{(\bar{u},\bar{x}+\ee h)}(t),\bar{u}(t)) < 0$ for all $t< t_{\ee,h}$ and $\max_{i=1,\ldots, p}g_{i}(x^{(\bar{u},\bar{x}+\ee h)}(t_{\ee,h}),\bar{u}(t_{\ee,h})) \geq 0$.
	Taking an arbitrary $\sigma \in ]0, t_{\ee,h}[$ and setting $\xi_{\ee,h} \triangleq x^{(\bar{u},\bar{x}+\ee h)}(\sigma)$, we indeed have $\xi_{\ee,h} \in G_{-}$. Assume, by contradiction, that there exists $\tilde{u}\in \UU$ such that $\max_{i=1,\ldots,p}g_{i}(x^{(\tilde{u},\xi_{\ee,h})}(t),\tilde{u}(t)) < 0$ for all $t \in [\sigma, \sigma + \tau[$ for some sufficiently small $\tau >0$ and $\zeta \triangleq x^{(\tilde{u},\xi_{\ee,h})}(\sigma + \tau) \in \Int(\AAA)$, which indeed implies that $x^{(\tilde{u},\xi_{\ee,h})}(t) \in G_{-}$ for all $t\in  [\sigma, \sigma + \tau[$.
	As a consequence of (\ref{AAAopt}) and (\ref{DAopt}), there exists $v\in \UU$ such that 
	$\ds \esup_{t\in [0,\infty)}\max_{i=1,\ldots,p} g_{i}(x^{(v,\zeta)}(\tau + \sigma + t),v(\tau + \sigma + t)) < 0$. 
	Setting $\tilde{v}= \bar{u} \Join_{\tau} \tilde{u} \Join_{\tau+\sigma}v$, we easily verify that 
	$\ds \esup_{t\in [0,\infty)}\max_{i=1,\ldots, p} g_{i}(x^{(\tilde{v},\bar{x}+\ee h)}(t),\tilde{v}(t)) <0$, 
	which implies, again by (\ref{AAAopt}) and (\ref{DAopt}), that $\bar{x}+\ee h \in \Int(\AAA)$, the whole integral curve $x^{(\tilde{v},\bar{x}+\ee h)}$ remaining in $G_{-}$, hence contradicting the fact that $\bar{x}+\ee h\in \AC$. We thus conclude that no integral curve starting in $\AC$ can penetrate the interior of $\AAA$ before leaving $G_{-}$. 
	Note that along the same lines and taking the limit as $\ee, \Vert h\Vert \rightarrow 0$, we prove the same result for $\DAM$ (see Corollary \ref{bar-sem-cor}).\\
	Note that $x\in G_{-}$ is equivalent to $\Int(U(x)) \neq \emptyset$. Thus, because $\DAM \subset G_-$, we can conclude that $\bar{x}\in \DAM$ implies $x^{(\bar{u},\bar{x})}(t) \in \DAM$ for all $t$ such that 
	$\Int(U(x^{(\bar{u},\bar{x})}(t))) \neq \emptyset$. It results that 
$$ \bbar{t} \triangleq \sup \{ t\in [0,\infty) : x^{(\bar{u},\bar{x})}(t)\in \DAM\}  = \sup \{ t\in [0,\infty) : U(x^{(\bar{u},\bar{x})}(t)) \neq\emptyset,  x^{(\bar{u},\bar{x})}(t)\in \DAM \} $$
	Thus $\bbar{t}$ satisfies 
	$$\min_{v\in U} \max_{i=1,\ldots,p} g_{i}(x^{(\bar{u},\bar{x})}(\bbar{t}),v) = \max_{i=1,\ldots, p} g_{i}(x^{(\bar{u},\bar{x})}(\bbar{t}),\bar{u}(\bbar{t})) = 0, \quad x^{(\bar{u},\bar{x})}(\bbar{t}) \in \cl(\DAM)$$
	which proves that the arc of integral curve $x^{(\bar{u},\bar{x})}$ intersects $G_{0}$.
	\end{proof}

In the course of the proof of Proposition~\ref{boundary:prop}, we have proven the following result which is of interest by itself (\emph{semi-permeability}):
\begin{cor}\label{bar-sem-cor}
Assume (A1) to (A5) hold. Then from any point on the boundary $\DAM$, there cannot exist a trajectory penetrating the interior of $\AAA$ before leaving $G_{-}$.
\end{cor}
\subsection{Ultimate Tangentiality}\label{subsec:UltimateTangen}
We now characterise the intersection of $\DAM$ with $G_0$ at the point $z$ defined in Proposition~\ref{boundary:prop}. We define
\begin{equation}\label{gtilde}
\tg(x)\triangleq \min_{u\in U} \max_{i\in\{1,\dots,p\}}g_i(x,u).
\end{equation}
Comparing to \eqref{def:G0} we immediately see that $G_0$ is the set of points $x\in G$ such that $\tg(x)=0$.
We prove that $\tilde{g}$ is locally Lipschitz, a simplified version of a result of J. Danskin \cite{danskin}:
\begin{lem}\label{tildegLip}
The function $\tilde{g}$ is locally Lipschitz, and thus absolutely continuous and almost everywhere differentiable, on every open and bounded subset of $\RR^{n}$.
\end{lem}
\begin{proof}
Consider the family of subsets of $\bigcap_{i=1,\ldots, p} \cl(g_{i}^{-1}(]-\infty, 0]))$ defined by 
$$\mathcal{O}_j \triangleq \{(x,u) \in \bigcap_{i=1,\ldots, p} \cl(g_{i}^{-1}(]-\infty, 0])) :\max_{i = 1,\dots,p} g_i(x,u) = g_j (x,u)\}, \quad j= 1,\ldots,p.$$ 
It is clear that $\bigcup_{j=1,\dots,p} \mathcal{O}_j = \bigcap_{i=1,\ldots, p} \cl(g_{i}^{-1}(]-\infty, 0]))$ and that we can extract a minimal subfamily of $\{\mathcal{O}_j\}$ still covering $\bigcap_{i=1,\ldots, p} \cl(g_{i}^{-1}(]-\infty, 0]))$, where every $\mathcal{O}_j$ has non-empty interior. In the sequel we only consider this subfamily. 
Given $x_1$ and $x_2$ in $G_{-}$ arbitrarily close, there exists $i_1$ such that $(x_1, u_1)\in \mathcal{O}_{i_{1}}$ with $u_1$ such that $g_{i_{1}}(x_1,u_1) = \min_{u\in U} g_{i_{1}}(x_1,u)$, and such that $(x_2, u_1)\in \mathcal{O}_{i_{1}}$.
Thus, we get
\begin{equation}
\begin{aligned}
\tilde{g}(x_2) - \tilde{g}(x_1) & =  \min_{u\in U} \max_{i\in\{1,\dots,p\}}g_i(x_2,u) - \min_{u\in U} \max_{i\in\{1,\dots,p\}}g_i(x_1,u)  \\
& \leq  \max_{i\in\{1,\dots,p\}}g_i(x_2,u_1) -\max_{i\in\{1,\dots,p\}}g_i(x_1,u_1) \\
&\leq \max_{i\in\{1,\dots,p\}}g_i(x_2,u_1) - g_{i_{1}}(x_1,u_1)\\
& \leq  g_{i_{1}}(x_2,u_1) - g_{i_{1}}(x_1,u_1)\label{BVr}
\end{aligned}
\end{equation}

Thus, since $g$ is continuously differentiable in $x$ for all $u$, there exists a point $\xi_1$ such that  $g_{i_{1}}(x_2,u_1) -g_{i_{1}}(x_1,u_1) =D_x g_{i_{1}}(\xi_1, u_1) \left( x_2 - x_1 \right)$.

Similarly, there exists $i_{2}$ such that $(x_2,u_2)\in \mathcal{O}_{i_{2}}$ with  $g_{i_{2}}(x_2,u_2) = \min_{u\in U} g_{i_{2}}(x_2,u)$ and $(x_{1},u_{2})\in \mathcal{O}_{i_{2}}$. We get
\begin{equation}\label{BVl}
g_{i_{2}}(x_2,u_2) -g_{i_{2}}(x_1,u_2) \leq \tilde{g}(x_2) - \tilde{g}(x_1)
\end{equation}
Again, there exists a point $\xi_2$ such that  $g_{i_{2}}(x_2,u_2) -g_{i_{2}}(x_1,u_2) = D_x g_{i_{2}}(\xi_2, u_2) \left( x_2 - x_1 \right)$. Combining \eqref{BVr} and \eqref{BVl}
yields
$$\vert \tilde{g}(x_2) - \tilde{g}(x_1) \vert \leq C \Vert x_2 - x_1\Vert$$
with $C = \sup(\Vert D_x g_{i_{1}}(\xi_1, u_1) \Vert, \Vert D_x g_{i_{2}}(\xi_2, u_2) \Vert)$. It results that $\tilde{g}$ is locally Lipschitz. The absolute continuity and almost everywhere differentiability follow from Rademacher's theorem (see e.g. \cite[Theorem 3.1]{heinonen}. See also \cite{Clarke,Clarke_et_al_springer}), which achieves to prove the lemma.
\end{proof}

We summarise a few concepts from nonsmooth analysis \cite{Clarke_et_al_springer}  that will be used in the next proposition. Consider $h: X \rightarrow \mathbb{R}$, where $X$ is a finite dimensional vector space, and $h$ is Lipschitz with Lipschitz constant $K$ near a given point $x \in X$. The \emph{generalised directional derivative} of $h$ at $x$ in the direction $v$ is defined as follows:
\begin{equation}\label{eq:GenDer}
h^0(x;v) \triangleq \limsup_{y\rightarrow x,t\rightarrow 0 ^+} \frac{h(y + tv) - h(y)}{t}
\end{equation}

We also need to introduce the \emph{generalised gradient} of $h$ at $x$, labeled $\partial h(x)$. It is well-known that in our setting, where we consider a Lipschitz function $h:\mathbb{R}^n \rightarrow \mathbb{R}$, the generalised gradient is the compact and convex set:
\begin{equation}\label{eq:GerGradient}
\partial h(x) = \co\{\lim_{i\rightarrow \infty} Dh^T(x_i): x_i \rightarrow x, x_i \notin \Omega_1 \cup \Omega_2 \}
\end{equation}
where $Dh^T(x)$ denotes the transpose of the row vector $Dh(x)$ at $x$, $\Omega_1$ is a zero measure set where $h$ is nondifferentiable (recall that $h$ is differentiable almost everywhere), $\Omega_2$ is any zero-measure set and recall that $\co(S)$ denotes the closed and convex hull of an arbitrary set $S$. 
Equivalently, denoting by $B_{\ee}(x)$ the open ball of radius $\ee$ centered at $x$, we have:
$$
\partial h(x) =\bigcap_{\ee>0} \bigcap_{\meas(\Omega)=0} \co\left( Dh^T(B_{\ee}(x)\setminus\Omega)\right)
$$

The relationship between the generalised directional derivative and the generalised gradient is given by:
\begin{equation}\label{eq:MaxEquation}
h^0(x;v) = \max_{\xi\in\partial h(x)}  \xi^T v .
\end{equation}

\begin{pr}\label{ult-tan-1d-pr}
Assume (A1) to (A5) hold. Consider $\bar{x} \in \DAM$ and $\bar{u}\in \UU$ as in Proposition~\ref{boundary:prop}, i.e. such that the integral curve $x^{(\bar{u},\bar{x})}(t) \in \DAM$ for all $t$ in some time interval until it reaches $G_{0}$ at some finite time $\bar{t}\geq 0$. Then, the point $z= x^{(\bar{u},\bar{x})}(\bar{t})\in \cl(\DAM)\cap G_{0}$,  satisfies
\begin{equation}\label{nonsmoothUltimateTan}
0= \max_{\xi\in\partial\tilde{g}(z)} \xi^T f(z, \bar{u}(\bar{t})) = \min_{v\in U(z)} \max_{\xi\in\partial\tilde{g}(z)}  \xi^T  f(z, v) = \max_{\xi\in\partial\tilde{g}(z)} \min_{v\in U(z)}  \xi^T  f(z, v).
\end{equation}
Moreover, if the function $\tilde{g}$ is differentiable at the point $z$, then condition \eqref{nonsmoothUltimateTan} reduces to the smooth counterpart:
\begin{equation}\label{eq:smoothUltimateTan}
0 = L_{f}\tilde{g}(z,\bar{u}(\bar{t})) = \min_{u\in U(z)} L_{f}\tilde{g}(z,u)
\end{equation}
where $L_f\tilde{g}(x,u) \triangleq D\tilde{g}(x)f(x,u)$ is the Lie derivative of $\tilde{g}$ along the vector field $f$ at $(x,u)$.
\end{pr}

\begin{proof}
Let $x_0\in \DAM$, then there exists a $\bar{u}\in \UU$ such that $\tilde{g}(x^{(\bar{u},x_0)}(t)) < 0$ until $x^{(\bar{u},x_0)}$ intersects $G_0$ at some $\tilde{t}$. As in the proof of Proposition~\ref{boundary:prop}, we consider an open set ${\mathcal O} \subset \RR^n$ such that $x_0 +\ee h \in \AC$ for all $h \in {\mathcal O}$ and $\Vert h\Vert \leq H$, with $H$ arbitrarily small, and all $\ee$ sufficiently small.

Introduce a needle perturbation of $\bar{u}$, labeled $u_{\kappa,\ee}$, at some Lebesgue point $\tau$ of $\bar{u}$ before $x^{(\bar{u},x_0)}$ intersects $G_0$, in the spirit of \cite{DeDona_siam}, i.e. a variation $u_{\kappa,\ee}$ of $\bar{u}$, parameterized by the vector 
$$\kappa \triangleq (v,\tau,l) \in U(x^{(\bar{u},x_0 + \ee h)}(\tau-l\ee)) \times [0,T] \times [0,L]$$ with bounded $T,L$, 
of the form
\begin{equation}\label{needleu-eq}
u_{\kappa,\ee} \triangleq 
\bar{u} \Join_{(\tau-l\ee)} v \Join_{\tau} \bar{u}
=
\left\{ \begin{array}{lcl}
v&\mbox{\textrm{on}}& [\tau-l\ee, \tau[\\
\bar{u}&\mbox{\textrm{elsewhere on}}&[0,T]
\end{array}\right. 
\end{equation}
where $v$ stands for the constant control equal to $v \in U(x^{(\bar{u},x_0)}(\tau))$ for all $t\in [\tau-l\ee, \tau[$. 
Remark that, by definition of $G_-$ and $U(x)$, since $x^{(\bar{u},x_0)}(t) \in G_{-}$ for all $t < \tilde{t}$, we have $\bar{u}(t)\in U(x^{(\bar{u},x_0 )}(t))$ for all $t < \tilde{t}$ and thus $U(x^{(\bar{u},x_0 )}(t)) \not = \emptyset$ for all $t < \tilde{t}$. 

Because $x_0 + \ee h \in \AC$, $\exists t_{\ee,\kappa,h}<\infty$ at which  $x^{(u_{\kappa,\ee},x_0+\ee h)}(t_{\ee,\kappa,h})$ crosses $G_0$, see Proposition \ref{boundary:prop}. As a result of the uniform convergence of $x^{(u_{\kappa,\ee},x_0+\ee h)}$ to $x^{(\bar{u},x_0)}$, there exists a $\bar{t} \geq \tilde{t}$, s.t. $x^{(u_{\kappa,\ee},x_0+\ee h)}(t_{\ee,\kappa,h})\rightarrow x^{(\bar{u},x_0)}(\bar{t})$ as $\ee \rightarrow 0$ and, according to the continuity of $\tilde{g}$, we have
	$$
	\lim_{\ee \rightarrow 0} \tilde{g}(x^{(u_{\kappa,\ee},x_0+\ee h)}(t_{\ee,\kappa,h})) = 0 = \tilde{g}(x^{(\bar{u},x_0)}(\bar{t})).
	$$

Because $\tilde{g}(x^{(u_{\kappa,\ee},x_0+\ee h)}(t_{\ee,\kappa,h})) = 0$ and $\tilde{g}(x^{(\bar{u},x_0)}(t_{\ee,\kappa,h})) \leq 0$ (recall that $\tilde{g}(x^{(\bar{u},x_0)}(t_{\ee,\kappa,h})) \leq g(x^{(\bar{u},x_0)}(t_{\ee,\kappa,h}), \bar{u}(t_{\ee,\kappa,h})) \leq 0$ since the pair $(x^{(\bar{u},x_0)}(t),\bar{u}(t))$ satisfies the constraints for all $t$), we have that
$$
\tilde{g}(x^{(u_{\kappa,\ee},x_0+\ee h)}(t_{\ee,\kappa,h})) - \tilde{g}(x^{(\bar{u},x_0)}(t_{\ee,\kappa,h})) \geq 0.
$$

Recall from \cite{PBGM} as well as \cite{DeDona_siam} that
$$
x^{(u_{\kappa,\ee},x_0+\ee h)}(t_{\ee,\kappa,h}) = x^{(\bar{u},x_0)}(t_{\ee,\kappa,h}) + \ee w(t_{\ee,\kappa,h},\kappa,h) + O(\ee^2)
$$
where
$$
w(t,\kappa,h) \triangleq  \Phi^{\bar{u}}(t,0)h + l \Phi^{\bar{u}}(t,\tau) \left( f(x^{(\bar{u},x_0)}(\tau), v) - f(x^{(\bar{u},x_0)}(\tau), \bar{u}(\tau)) \right), 
$$ 
$\Phi^{\bar{u}}(t,s)$ being the solution to the variational equation at time $t$ starting from time $s$ (see equation \eqref{eq:VariationalEquation} in Appendix \ref{Appendix:PMP}), $\tau$ being any Lebesgue point of the control $\bar{u}$, with $v \in U(x^{(\bar{u},x_0)}(\tau))$ and where we have denoted by $O(\ee^{k})$  a continuous function of $\ee^{k}$ defined in a small open interval containing 0 and such that $\lim_{\ee\rightarrow 0} \frac{O(\ee^{k})}{\ee^{k-r}} = \lim_{\ee\rightarrow 0} O(\ee^{r}) = 0$ for all $0\leq r \leq k-1$, $k, r \in \NN$. 

Since $\tilde{g}$ is almost everywhere differentiable, we have:
\begin{equation}\label{genDerOfGTilde}
\frac{\tilde{g}(x^{(u_{\kappa,\ee},x_0+\ee h)}(t_{\ee,\kappa,h})) - \tilde{g}(x^{(\bar{u},x_0)}(t_{\ee,\kappa,h}))}{\ee} =D \tilde{g}(x^{(\bar{u},x_0)}(t_{\ee,\kappa,h})) . w(t_{\ee,\kappa,h},\kappa,h)) +O(\ee) \geq 0
\end{equation}
for every $v \in U(x^{(\bar{u},x_0)}(\tau))$ and almost every $\ee$ and $h$.

If we take any accumulation point of the right-hand side of \eqref{genDerOfGTilde} as $\ee$ and $\Vert h\Vert$ tend to zero, according to \eqref{eq:GenDer} and \eqref{eq:MaxEquation}, we get, after division by $l$:
\begin{equation}\label{genDerOfGTilde2}
 \xi^T\Phi^{\bar{u}}(\bar{t},\tau) \left( f(x^{(\bar{u},x_0)}(\tau), v) - f(x^{(\bar{u},x_0)}(\tau), \bar{u}(\tau)) \right) \geq 0 \quad \forall \xi\in \partial \tilde{g}(x^{(\bar{u},x_0)}(\bar{t}))
\end{equation}
Assume for a moment that we can replace $v$ in \eqref{genDerOfGTilde2} by a continuous family $v_{\tau}$ with respect to $\tau$ such that $\lim_{\tau \rightarrow \bar{t}} v_{\tau} = v$. This result is proven in Lemma \ref{lemma:convergenceVtau} below. Thus, taking the limit as $\tau$ tends to $\bar{t}$ in \eqref{genDerOfGTilde2}, we get
\begin{equation}\label{genDerOfGTildeLim}
\xi^T\left( f(z, v) - f(z, \bar{u}(\bar{t}) \right) \geq 0, \quad \forall \xi\in\partial\tilde{g}(z), \quad  \forall v\in U(z)
\end{equation}
where $z = x^{(\bar{u},x_{0})}(\bar{t})$. Therefore,
\begin{equation}\label{maxminBarU}
\max_{\xi\in\partial\tilde{g}(z)} \xi^Tf(z, \bar{u}(\bar{t})) = \min_{v\in U(z)} \max_{\xi\in\partial\tilde{g}(z)}  \xi^Tf(z, v).
 \end{equation}
Since the mapping $\xi\mapsto \xi^Tf(z,v)$ is linear on the compact and convex set $\partial\tilde{g}(z)$ and the mapping $v\mapsto \xi^Tf(z,v)$ is convex and continuous on the compact set $U(z)$ which is convex by (A.5), it results from the minimax theorem of Von Neumann (see e.g. \cite{berge}) that 
\begin{equation}\label{minimax}
 \min_{v\in U(z)} \max_{\xi\in\partial\tilde{g}(z)}  \xi^Tf(z, v) =  \max_{\xi\in\partial\tilde{g}(z)} \min_{v\in U(z)} \xi^Tf(z, v).
 \end{equation}

If $\bar{t}$ is not an L-point, it suffices to modify $\bar{u}$ on the 0-measure set $\{\bar{t}\}$ by replacing $\bar{u}(\bar{t})$ by its left limit $\bar{u}(\bar{t}_{-})$ in the latter expression.

We will now show that this expression is equal to 0. 
On the one hand, because $\tilde{g}$ is locally Lipschitz, $D\tilde{g}$ exists almost everywhere and the mapping $t\mapsto \tilde{g}(x^{(\bar{u},x_{0})}(t))$ is nondecreasing on some small interval $(\bar{t} - \eta, \bar{t}]$ with $\eta>0$ sufficiently small, and we have $D\tilde{g}(x^{(\bar{u},x_{0})}(t)) . f(x^{(\bar{u},x_{0})}(t),\bar{u}(\bar{t})) \geq 0$ where $D\tilde{g}$ exists.
Therefore we conclude that 
\begin{equation}\label{tildeg0geq0}
\tilde{g}^0(z;f(z,\bar{u}(\bar{t}_{-}))) \geq 0.
\end{equation}

On the other hand, by definition, we have:
$$\begin{aligned}
0=&\frac{\tilde{g}(x^{(u_{\kappa,\ee},x_0+\ee h)}(t_{\ee,\kappa,h})) - \tilde{g}(x^{(\bar{u},x_{0})}(\bar{t}))}{\ee}\\
=&\left[ \frac{\tilde{g}(x^{(u_{\kappa,\ee},x_0+\ee h)}(t_{\ee,\kappa,h})) - \tilde{g}(x^{(\bar{u},x_{0})}(t_{\ee,\kappa,h}))}{\ee} \right] \ +\left[ \frac{\tilde{g}(x^{(\bar{u},x_{0})}(t_{\ee,\kappa,h})) - \tilde{g}(x^{(\bar{u},x_{0})}(\bar{t}))}{\ee}\right]
\end{aligned}
$$
Thus, since the first bracketed term of the right-hand side has been proven to be $\geq 0$, we immediately get
$$
\limsup_{\ee\rightarrow 0_{+}} \frac{\tilde{g}(x^{(\bar{u},x_{0})}(\bar{t}))  - \tilde{g}(x^{(\bar{u},x_{0})}(t_{\ee,\kappa,h}))}{\ee} = \limsup_{\ee\rightarrow 0_{+}} \frac{\tilde{g}(x^{(u_{\kappa,\ee},x_0+\ee h)}(t_{\ee,\kappa,h})) - \tilde{g}(x^{(\bar{u},x_{0})}(t_{\ee,\kappa,h}))}{\ee} \geq 0
$$
But, since 
$$
 \limsup_{\ee\rightarrow 0_{+}} \frac{\tilde{g}(x^{(\bar{u},x_{0})}(\bar{t}))  - \tilde{g}(x^{(\bar{u},x_{0})}(t_{\ee,\kappa,h}))}{\ee} = -\tilde{g}^0(z;f(z,\bar{u}(\bar{t}_{-})))
$$
we conclude that $-\tilde{g}^0(z;f(z,\bar{u}(\bar{t}_{-})))\geq 0$. Comparing to \eqref{tildeg0geq0}, we get $\tilde{g}^0(z;f(z,\bar{u}(\bar{t}_{-})))= 0$, or according to \eqref{eq:MaxEquation}:
$$
0 = \max_{ \xi\in\partial\tilde{g}(z)} \xi^Tf(z, \bar{u}(\bar{t})) 
$$
which, together with \eqref{maxminBarU} and \eqref{minimax}, proves \eqref{nonsmoothUltimateTan}.

If $\tilde{g}$ is differentiable at $z$, we can apply exactly the same argument as before up until equation \eqref{genDerOfGTilde}. Thus, letting $\Vert h\Vert \rightarrow 0$ and dividing by $l$, we get:
$$D\tilde{g}(x^{(\bar{u},x_0)}(t_{\ee,\kappa,h})) . \left[ \Phi^{\bar{u}}(t_{\ee,\kappa,h},\tau) \left( f(x^{(\bar{u},\bar{x})}(\tau), v) - f(x^{(\bar{u},\bar{x})}(\tau), \bar{u}(\tau)) \right)\right] + O(\ee) \geq 0.$$
If $\ee$ now tends to zero we get
$$D \tilde{g}(z) \Phi^{\bar{u}}(\bar{t},\tau) f(x^{(\bar{u},\bar{x})}(\tau),v) \geq D\tilde{g}(z) \Phi^{\bar{u}}(\bar{t},\tau) f(x^{(\bar{u},\bar{x})}(\tau),\bar{u}(\tau)), \quad \forall v\in U(x^{(\bar{u},\bar{x})}(\tau)).$$
We again assume that $\bar{t}$ is an L-point for the control $\bar{u}$, and construct the same continuous mapping $\tau \mapsto v_{\tau}$ as before, such that $\lim_{\tau\rightarrow \bar{t}} v_{\tau} = v$, for an arbitrary $v\in U(z)$ to get:
$$D \tilde{g}(z) f(z,v) \geq D\tilde{g}(z) f(z,\bar{u}(\bar{t})), \quad \forall v\in U(z)$$
or, using the Lie derivative notation:
$$L_{f}\tilde{g}(z,\bar{u}(\bar{t})) = \min_{v\in U(z)} L_{f}\tilde{g}(z,v).$$

Interpreting $L_{f}\tilde{g}(z,\bar{u}(\bar{t}))$ as the time derivative of $t\mapsto \tilde{g}(x^{(\bar{u},x_{0})}(t))$ and remarking that the latter mapping is non decreasing on an interval $]\bar{t}-\eta, \bar{t}]$, for some $\eta > 0$ small enough, we indeed deduce that $L_{f}\tilde{g}(z,\bar{u}(\bar{t})) \geq 0$. The same mapping being non increasing on the interval $[\bar{t},\bar{t}+\eta'[$, we have $L_{f}\tilde{g}(z,\bar{u}(\bar{t})) \leq 0$, which finally proves that 
$L_{f}\tilde{g}(z,\bar{u}(\bar{t}))=0$. If $\bar{t}$ is not an L-point of $\bar{u}$, the same modification of $\bar{u}$ at $\bar{t}$, as in the nonsmooth case, may be applied, which achieves to prove the proposition.
\end{proof}

\begin{lem}\label{lemma:convergenceVtau}
Under the assumptions of Proposition~\ref{ult-tan-1d-pr}, for all $v\in U(z)$ with $z = x^{(\bar{u},x_0)}(\bar{t})$, there exists a continuous mapping $\tau \mapsto v_{\tau}$ from $[\bar{t} - \eta,\bar{t}[$ to $U$, with $\eta > 0$ small enough, such that $v_{\tau} \in U(x^{(\bar{u},x_0)}(\tau))$ for all $\tau \in [\bar{t} - \eta,\bar{t}[$ and $\ds \lim_{\tau \nearrow \bar{t}} v_{\tau} = v$.
\end{lem}

\begin{proof}
	Recall that the condition $v_{\tau} \in U(x^{(\bar{u},x_0)}(\tau))$ is equivalent to $g(x^{(\bar{u},x_0)}(\tau), v_{\tau}) \preceq 0$ for all $\tau\in [\bar{t} - \eta,\bar{t}[$ and, since $z= x^{(\bar{u},x_0)}(\bar{t})\in G_0$, $v\in U(z)$ is such that $g(z,v) \circeq 0$. We construct such a $v_{\tau}$ as follows.
	
	Since, by assumption, $\#\II(z,\bar{u}(\bar{t}_-)) = s_1$ and $\#\JJ(\bar{u}(\bar{t}_-)) = s_2$, with $\max(s_1,s_2)>0$, consider the equation 
	\[
	\Gamma(x,u) = \left(\begin{array}{c}
	g_{i_1}(x,u)\\
	\dots \\
	g_{i_{s_1}}(x,u)\\
	\gamma_{j_1}(u)\\
	\dots\\
	\gamma_{j_{s_2}}(u)
	\end{array}\right) = 0.
	\]	
	According to assumption (A4) and the implicit function theorem, there exists a continuously differentiable mapping:
	
	$$\hat{u} \triangleq (\hat{u}_1,\dots,\hat{u}_{s_1 + s_2}):\RR^n \times \RR^{m - (s_1 + s_2)} \rightarrow \RR^{s_1 + s_2}$$ 
	defined in a neighbourhood of the point $(z,v_{s_1+s_2+1},\dots,v_{m})$, labelled $\Nnn$, 
	such that	
	$$(\hat{u}(x,v_{s_1 + s_2 + 1},\dots,u_{m}),v_{s_1 + s_2 + 1},\dots,v_{m}) = v$$ 
	and
	$$\Gamma(x,\hat{u}(x,u_{s_1 + s_2 + 1},\dots,u_{m}),u_{s_1 + s_2 + 1},\dots,u_{m}) = 0 \quad \forall (x,u_{s_1 + s_2 +1},\dots,u_{m}) \in \Nnn.$$
	Then we define 
	$$v_{\tau} \triangleq \tilde{u}(x^{(\bar{u},x_0)}(\tau),v_{s_1 + s_2 + 1},\dots, v_m) \quad \forall \tau \in [\bar{t} - \eta,\bar{t}[$$
with $\eta$ small enough such that $(x^{(\bar{u},x_0)}(\tau),v_{s_1 + s_2 + 1},\dots, v_m)$ remains in $\Nnn$ in the whole interval $[\bar{t} - \eta,\bar{t}[$.
	Therefore, we have $\Gamma(x^{(\bar{u},x_0)}(\tau),v_{\tau}) = 0$ for all $\tau \in [\bar{t} - \eta,\bar{t}[$. Moreover, since $v_{\tau}$ so defined is clearly a continuous function of $\tau$, and since, by assumption (A.4), $\eta$ may be possibly decreased in order that  
	$$g_i (x^{(\bar{u},x_0)}(\tau),v_{\tau},v_{s_1+s_2 + 1},\dots, v_m) < 0  \quad  \forall \tau \in [\bar{t} - \eta,\bar{t}[, \quad \forall i\not\in \II(z,\bar{u}(\bar{t}_-))$$ and 
	$$\gamma_j (x^{(\bar{u},x_0)}(\tau),v_{\tau},v_{s_1+s_2 + 1},\dots, v_m) < 0 \quad \forall \tau \in [\bar{t} - \eta,\bar{t}[, \quad \forall j \not\in \JJ(\bar{u}(\bar{t}_-))$$ 
	we have, as required, $v_{\tau}\in U(x^{(\bar{u},x_0)}(\tau))$ and $\lim_{\tau\nearrow \bar{t}} v_{\tau} = v$.
\end{proof}

\section{The Barrier Equation}\label{sec:BarrierEquation}
~

We next present the main result of the paper, Theorem \ref{BarrierTheorem}, which gives necessary conditions satisfied by an integral curve running along the barrier. The proof of the theorem utilises the maximum principle for problems with mixed constraints stated in terms of reachable sets where the extremal curves are those whose endpoints at each time $t$ belong to the boundary of the reachable set at the same instant of time. See the Appendix \ref{Appendix:PMP} for more details.

\begin{thm}\label{BarrierTheorem}
Under the assumptions of Proposition \ref{boundary:prop}, consider an integral curve $x^{\bar{u}}$ on $\DAM \cap \cl(\Int(\AAA))$ and assume that the control function $\bar{u}$ is piecewise continuous. Then $\bar{u}$ and $x^{\bar{u}}$ satisfy the following necessary conditions.

There exists a non-zero absolutely continuous adjoint $\lambda^{\bar{u}}$ and piecewise continuous multipliers $\mu_i^{\bar{u}} \geq 0$, $i=1,\dots,p$, such that:
\begin{equation}\label{costateEquation}
\dot{\lambda}^{\bar{u}}(t) = -\left(\frac{\partial f}{\partial x}(x^{\bar{u}}(t),\bar{u}(t)) \right)^T \lambda^{\bar{u}}(t) - \sum_{i=1}^{p}\mu_i^{\bar{u}}(t)\frac{\partial g_i}{\partial x}(x^{\bar{u}}(t),\bar{u}(t))
\end{equation}
with the ``complementary slackness condition''
\begin{equation}\label{eq:complement}
\mu_i^{\bar{u}}(t)g_i(x^{\bar{u}}(t),\bar{u}(t)) = 0, \quad i=1,\ldots, p
\end{equation}
and final conditions
\begin{equation}\label{eq:finalConditions}
\lambda^{\bar{u}}(\bar{t})^T \in \arg \max_{\xi\in\partial \tilde{g}(z)} \xi. f(z,\bar{u}(\bar{t}))
\end{equation} 
where $z = x^{\bar{u}}(\bar{t})$ with $\bar{t}$ such that $z\in G_0$, i.e. $\min_{u\in U}\max_{i=1,\dots,p} g_i(z,u) = 0$, $\partial \tilde{g}(z)$ being the generalised gradient of $\tilde{g}$ defined by \eqref{gtilde} at $z$.

Moreover, at almost every $t$, the Hamiltonian, $H(\lambda^{\bar{u}}(t),x^{\bar{u}}(t),u) = \left(\lambda^{\bar{u}}(t) \right)^Tf(x^{\bar{u}}(t),u)$, is minimised over the set $U(x^{\bar{u}}(t))$ and equal to zero:
\begin{equation}\label{HamiltonianMinimised}
\begin{aligned}
\min_{u\in U(x^{\bar{u}}(t))} \lambda^{\bar{u}}(t)^T f(x^{\bar{u}}(t),u) &=\min_{u\in U} \left[ \left(\lambda^{\bar{u}}(t) \right)^T f(x^{\bar{u}}(t),u) + \sum_{i=1}^{p} \mu_{i}^{\bar{u}}(t)g_{i}(x^{\bar{u}}(t),u)\right] \\
&= \lambda^{\bar{u}}(t)^T f(x^{\bar{u}}(t),\bar{u}(t)) = 0
\end{aligned}
\end{equation}

\end{thm}
\begin{rem}

To compute \eqref{HamiltonianMinimised} the following necessary conditions are useful:

\begin{equation}\label{eq:HamiltonianKKT}
\left\{\begin{array}{l}
\ds H(\lambda^{\bar{u}}(t),x^{\bar{u}}(t),\bar{u}(t)) = 0\\
\ds \frac{\partial H}{\partial u}(\lambda^{\bar{u}}(t),x^{\bar{u}}(t),\bar{u}(t)) + \sum_{i=1}^{p}\mu_i^{\bar{u}}(t) \frac{\partial g_{i}}{\partial u}(x^{\bar{u}}(t),\bar{u}(t)) + \sum_{j=1}^{r}\nu_j^{\bar{u}}(t) \frac{\partial \gamma_{j}}{\partial u}(\bar{u}(t)) = 0 \\
\ds \mu_i^{\bar{u}}(t)g_i(x^{\bar{u}}(t),\bar{u}(t)) = 0, \quad \mu_i^{\bar{u}}(t) \geq 0 \quad i=1,\ldots, p \vspace{1em} \\
\ds \nu_j^{\bar{u}}(t)\gamma_j(\bar{u}(t)) = 0, \quad \nu_j^{\bar{u}}(t) \geq 0\quad j=1,\ldots, r.
\end{array}\right.
\end{equation}
\end{rem}

Before proving Theorem \ref{BarrierTheorem} we need to introduce the following definition: 
\begin{defn}\label{def:ConstrainedReachSet}
	The \emph{constrained reachable set} at time $t$ from initial condition $\bar{x}$ is given by:
	$$R_{t}(\bar{x}) \triangleq \{x\in \RR^n : \exists u \in \UU~ s.t.~x=x^{(u,\bar{x})}(t),~ g(x^{(u,\bar{x})}(s),u(s))\preceq 0~\mathrm{for~} \mathit{a.e.~} s\leq t \}$$
\end{defn}
\begin{lem}\label{att-boundary:lemma}
Let $\bar{x}\in\DAM \cap \cl(\Int(\AAA))$ and $\bar{u}\in \UU$ as in Proposition \ref{boundary:prop}, i.e. such that $x^{(\bar{u},\bar{x})}(t)\in\DAM$ for all $t\in[0,\bar{t}[$ where $\bar{t}$ is the time such that $\tilde{g}(x^{(\bar{u},\bar{x})}(\bar{t})) = 0$. Then, $x^{(\bar{u},\bar{x})}(t)\in \partial R_t(\bar{x})$ for all $0\leq t<\bar{t}$.
\end{lem}

\begin{proof}
We first prove that $R_{t}(\bar{x}) \subset \cl (\AC)$ for all $0\leq t<\bar{t}$. Assume by contradiction that for some $0 \leq t < \bar{t}$ we have $R_{t}(\bar{x})\cap \Int(\AAA)\neq \emptyset$. Then $\exists u\in \mathcal{U}$ such that $x^{(u,\bar{x})}(t)\in \Int(\AAA)$ for some $0\leq t<\bar{t}$, which contradicts the fact that $\bar{x} \in \DAM$ by Corollary~\ref{bar-sem-cor}, hence $R_{t}(\bar{x}) \subset \cl (\AC)$.

By complementarity $\Int(\AAA)\subset R_t(\bar{x})^{\mathsf C}$, and thus $\cl(\Int(\AAA))\subset \cl(R_t(\bar{x})^{\mathsf C})$. Thus, assume that $\bar{x}\in \DAM \cap \cl(\Int(\AAA))$ and that there exists $\bar{u}\in \UU$ as in Proposition \ref{boundary:prop}. Then it can be shown as in the proof of Corollary~\ref{bar-sem-cor} that there exists a sequence $\{x_k\}_{k\in \NN}$, with $x_k \in \Int(\AAA)$, and a sequence $\{ u_k\}_{k\in \NN}$, $u_k \in \UU$, such that every integral curve $x^{(u_{k},x_{k})}$ lies in $\Int(\AAA)$ and the sequence $\{x^{(u_{k},x_{k})}\}_{k}$ converges uniformly to $x^{(\bar{u},\bar{x})}$ on every compact interval $[0,T]$. We therefore immediately deduce that $x^{(\bar{u},\bar{x})}(t)\in\DAM \cap \cl(\Int(\AAA))$ for all $t < \bar{t}$ and hence that  $x^{(\bar{u},\bar{x})}(t)\in  \cl(R_t(\bar{x})^{\mathsf C})$. But because $x^{(\bar{u},\bar{x})}(t)\in R_t(\bar{x})$, and since $\partial R_t(\bar{x}) = R_t(\bar{x}) \cap \cl(R_t(\bar{x})^{\mathsf C})$, we conclude that $x^{(\bar{u},\bar{x})}(t)\in \partial R_t(\bar{x})$.
\end{proof}

\begin{proof}[Proof of Theorem \ref{BarrierTheorem}]

By Lemma \ref{att-boundary:lemma} we know that $x^{(\bar{u},\bar{x})}(t)\in \partial R_t(\bar{x})$ for all $0\leq t<\bar{t}$. Therefore, according to Theorem \ref{extrem:thm}, we know that $\bar{u}$ must satisfy \eqref{barriercond-eta}. Then, setting $\lambda^{\bar{u}} = -\eta^{\bar{u}}$ we get \eqref{costateEquation} with \eqref{eq:complement} and that the resulting dualised Hamiltonian $\tilde{\mathcal{H}}(x,u,\lambda,\mu) \triangleq \mathcal{H}(x,u,-\eta,\mu)$, defined by \eqref{eq:HamiltonianDef}, now must be minimised. Now taking the final conditions for $\lambda^{\bar{u}}$ as in Proposition \ref{ult-tan-1d-pr}, namely \eqref{nonsmoothUltimateTan}, we immediately deduce that at time $\bar{t}$ the minimised Hamiltonian must be zero, and thus the constant of \eqref{barriercond-eta} is equal to zero. Finally, according to the complementary slackness condition, \eqref{eq:CompSlackCond}, the minimisation of $\tilde{\mathcal{H}}$ becomes equivalent to \eqref{HamiltonianMinimised} which achieves the proof of the theorem.
\end{proof}

\begin{rem}
	If $\tilde{g}$ is differentiable at the point $z$, condition \eqref{eq:finalConditions} reduces to its smooth counterpart, i.e.,	$\lambda^{\bar{u}}(\bar{t})^T = D\tilde{g}(z) $
\end{rem}

\begin{rem}
	The assumption that $x^{(\bar{u},\bar{x})}\in \DAM \cap \cl(\Int(\AAA))$ means that we possibly miss isolated trajectories which are in $\AAA \setminus \cl(\Int(\AAA))$. The existence and computation of such trajectories, if they exist, are open questions.
\end{rem}

\section{Examples}\label{sec:Examples}

\subsection{Constrained Spring 1}\label{subsec:ConstrainedSpring1}
~

Consider the following constrained mass-spring-damper model:
\[
\left(
\begin{array}{c}
\dot{x}_{1}\\
\dot{x}_{2}
\end{array}
\right)
=
\left(
\begin{array}{cc}
0 & 1\\
-2 & -2
\end{array}
\right)
\left(
\begin{array}{c}
x_{1}\\
x_{2}
\end{array}
\right)
+\left(
\begin{array}{c}
0\\
1
\end{array}
\right)u
,\,\,\,\,\,|u|\leq 1, \,\,\,\,\,x_2 - u \leq 0
\]
{
\flushleft
where $x_1$ is the mass's displacement. The spring stiffness is here equal to 2 for a mass equal to 1 and the friction coefficient is equal to 2. $u$ is the force applied to the mass.
}

We identify $g(x,u) = x_2 - u$, $U = [-1,1]$ and $\tilde{g}(x) = x_2 - 1$. We also identify the following sets: $G = \{x\in\mathbb{R}^2:x_2\leq 1\}$, $G_{0} = \{x\in G: x_2 = 1\}$ and $U(x) = \{u\in U : x_2\leq u \leq 1\}$. Note that if $z \triangleq (z_1,z_2)\in G_0$, i.e. $z_2 = 1$, then $U(z)$ is the singleton $U(z) = \{1\}$. 

We have $\partial \tilde{g}(z) = \{ (0,1)^T \} = D\tilde{g}(z)^T$ (which means that $\tilde{g}$ is differentiable everywhere) and the ultimate tangentiality condition reads:
\[
\min_{u\in U(z)} D \tilde{g}(z)^T f(z,u) = 0
\]
which gives
\[
\min_{u\in U(z)} -2z_1 - 2z_2 + u = -2z_1 - 2 + 1 = 0
\]
Thus $z = (-\frac{1}{2},1)$. 

Let us now compute $\lambda(\bar{t})$. From \eqref{eq:finalConditions}, which here reduces to \eqref{eq:smoothUltimateTan}, we get that $\lambda(\bar{t}) = D\tilde{g}(z) = (0,1)$.

We now construct the barrier by integrating backwards from $z$ and $\lambda(\bar{t})$. From the minimisation of the Hamiltonian, $H(x,\lambda,u) = \lambda_1x_2 + \lambda_2(-2x_1 - 2x_2 + u)$, condition \eqref{HamiltonianMinimised}, we find that the control $\bar{u}$ associated with the barrier is given by
\[
\min_{x_2\leq u \leq 1} \lambda_1x_2 + \lambda_2(-2x_1 - 2x_2 + u) = 0
\]
which gives:
\[
\begin{array}{lll}\mathrm{if}~\lambda_2(t) < 0 & \\
&\bar{u}(t)=1 \\
\mathrm{if}~\lambda_2(t) > 0 & \\
&\bar{u}(t)=\left\{
\begin{array}{lll}
x_2 & \mathrm{if} & x_2 \in ]-1,1]\\
-1 & \mathrm{if} & x_2 \in ]-\infty,-1]
\end{array}\right.\\
\mathrm{if}~\lambda_2(t) = 0 & \\
&\bar{u}(t)=\mathrm{arbitrary}
\end{array}
\]

We note from condition \eqref{costateEquation} that if the constraint is active (i.e. $g(x,u) = 0$), the costate differential equation is given by
\[
\dot{\lambda}^{\bar{u}} = -\frac{\partial f}{\partial x}^T\lambda^{\bar{u}} - \mu^{\bar{u}}\frac{\partial g}{\partial x} =  
\left(
\begin{array}{cc}
0 & 2\\
-1 & 2
\end{array}\right)\lambda^{\bar{u}} - \mu^{\bar{u}}
\left(
\begin{array}{c}
0 \\
1
\end{array}\right)
\]
and is otherwise (when $g(x,u)<0$) given by
\begin{equation}\label{adjeq:nonactive}
\dot{\lambda}^{\bar{u}} = -\frac{\partial f}{\partial x}^T\lambda^{\bar{u}} = \left(
\begin{array}{cc}
0 & 2\\
-1 & 2
\end{array}\right)\lambda^{\bar{u}}.
\end{equation}

Recall that $\lambda_2(\bar{t})>0$ and $x_2(\bar{t}) > 0$. Therefore, because $\lambda$ and $x$ are continuous, $\bar{u}(t) = x_2(t)$ over an interval before $\bar{t}$. We can show that $\bar{u}(t) \neq 1$ over this interval: if $x_2 = 1$ and $u = 1$ over an interval before $\bar{t}$, then we get $\dot{x}_2 = -2x_1 - 2 + 1 = 0$ or $x_1 = -\frac{1}{2}$ which implies $\dot{x}_1 = 0$ for all $t\in]\bar{t} - \eta, \bar{t}],\,\,\eta >0$. However, we would also have $\dot{x}_1 = 1$ over $t\in]\bar{t} - \eta, \bar{t}]$, which contradicts the fact that $\dot{x}_1 = 0$ over this interval.

Therefore, only the constraint $g$ is active over an interval before $\bar{t}$, and by \eqref{eq:HamiltonianKKT}, we obtain $\mu$ over this interval:
\[
\frac{\partial H}{\partial u} + \mu\frac{\partial g}{\partial u} = \lambda_2 - \mu = 0
\]
and thus $\lambda_2 = \mu$. In addition the adjoint satisfies:
\begin{equation}\label{eq:ActiveConstraintDynamics}
\dot{\lambda} = \left(
\begin{array}{cc}
0 & 2\\
-1 & 1
\end{array}\right)\lambda,\quad \forall t\in]\bar{t} - \eta,\bar{t}]
\end{equation}

At some point in time before $\bar{t}$, let us label this point $\hat{t}$, we have $\lambda_2(\hat{t}) = 0$ and it can be verified that, at this time, $x_2(\hat{t}) = 0$ and $\lambda_1(\hat{t}) < 0$. Let us prove that $\lambda_2$ is negative on the interval $[0,\hat{t}]$. If $\lambda_2$ vanishes at some point in time, since we have
\[
\dot{\lambda}_2 = 2\lambda_1 = 0
\]
then $\lambda \equiv 0$ which contradicts our assertion. We conclude that over $[0,\hat{t}]$, $\lambda_2$ is either everywhere positive or everywhere negative.

If over this interval before $\hat{t}$ $\lambda_2 > 0$, then the co-state dynamics are as before, and $\dot{\lambda}_2 < 0$ which is equivalent to $-\lambda_1 + \lambda_2 < 0$, but this contradicts the fact that $\lambda_1(\hat{t}) < 0$. We can conclude that $\lambda_2$ is negative before $\hat{t}$, and that $\bar{u} = 1$ over this period. The costate dynamics are then given by \eqref{adjeq:nonactive}. The sign of $\lambda_2$ then remains negative until the trajectory intersects $G_0$ again. The barrier is shown in Figure \ref{ConstrainedSpringBarrier1}.

\begin{figure}[th]\label{ConstrainedSpringBarrier1}
\begin{center}
\includegraphics[width=0.7\columnwidth]{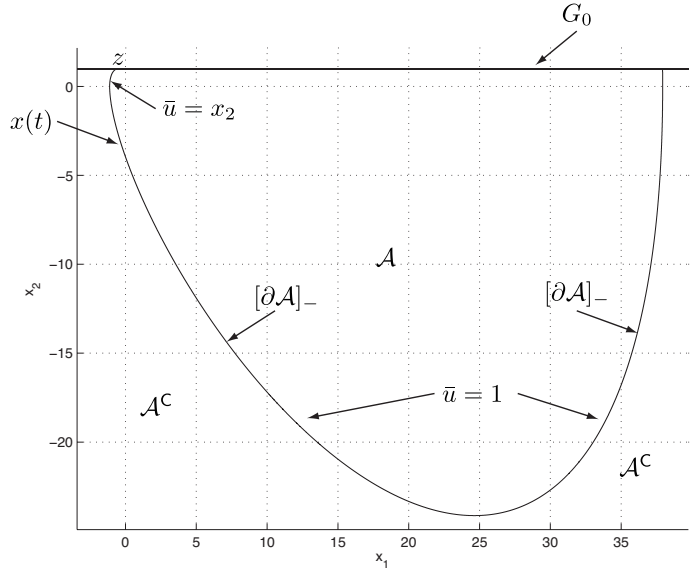}
\caption{Admissible set of the constrained spring from example \ref{subsec:ConstrainedSpring1}}
\end{center}
\end{figure}

\begin{rem}
	Note that Assumption (A4) does not hold true at the final point $z$ since there are two active constraints for only one control. However, since this condition is violated only at this point, we may conclude by continuity that condition \eqref{eq:finalConditions} still holds.
\end{rem}

\subsection{Constrained Spring 2}\label{subsec:ConstrainedSpring2}
~

Consider the same mass-spring-damper system with the same constants as in the previous example, but with a richer constraint:
\begin{equation}\label{diffeq:Ex2}
\left(
\begin{array}{c}
\dot{x}_{1}\\
\dot{x}_{2}
\end{array}
\right)
=
\left(
\begin{array}{cc}
0 & 1\\
-2 & -2
\end{array}
\right)
\left(
\begin{array}{c}
x_{1}\\
x_{2}
\end{array}
\right)
+\left(
\begin{array}{c}
0\\
1
\end{array}
\right)u
,\,\,\,\,\,|u|\leq 1, \,\,\,\,\,x_2(x_2 - u) \leq 0
\end{equation}

We identify $\tilde{g}(x) = x_2^2 - |x_2|$, and $G_0 = \{x:x_2 = 0 \cup x_2 = \pm 1\}$. $\tilde{g}$ is differentiable for $x_2 \neq 0$ and from \eqref{HamiltonianMinimised} and \eqref{eq:finalConditions} we identify, in same manner as in the previous example, two points of ultimate tangentiality, namely $z=(-\frac{1}{2},1)$ along with $\lambda(\bar{t}) = (0,1)$, and $z = (\frac{1}{2},-1)$ along with $\lambda(\bar{t}) = (0,-1)$. We defer the treatment of the $x_1$ axis, which is also in $G_0$, to the discussion below.

From the minimisation of the Hamiltonian, which is the same as in the previous example, we find the control $\bar{u}$:
\[
\begin{array}{lll}\mathrm{if}~\lambda_2(t) < 0 & \\
&\bar{u}(t)=\left\{
\begin{array}{lll}
1 & \mathrm{if} & x_2 \in ]0,1]\\
x_2 & \mathrm{if} & x_2 \in ]-1,0[
\end{array}\right. \\
\mathrm{if}~\lambda_2(t) > 0 & \\
&\bar{u}(t)=\left\{
\begin{array}{lll}
x_2 & \mathrm{if} & x_2 \in ]0,1]\\
-1 & \mathrm{if} & x_2 \in ]-1,0[
\end{array}\right.\\
\mathrm{if}~\lambda_2(t) = 0 & \\
&\bar{u}(t)=\mathrm{arbitrary}
\end{array}
\]

If we now integrate backwards from the points $(-\frac{1}{2},1)$ and $(\frac{1}{2},-1)$ with the control $\bar{u}(t)$ we obtain the barrier as in Figure \ref{Fig:ConstrainedSpringBarrier2}. It turns out that along both curves $\bar{u}(t) = x_2(t)$.

\begin{figure}[thpb]\label{Fig:ConstrainedSpringBarrier2}
	\begin{center}
		\includegraphics[width=0.7\columnwidth]{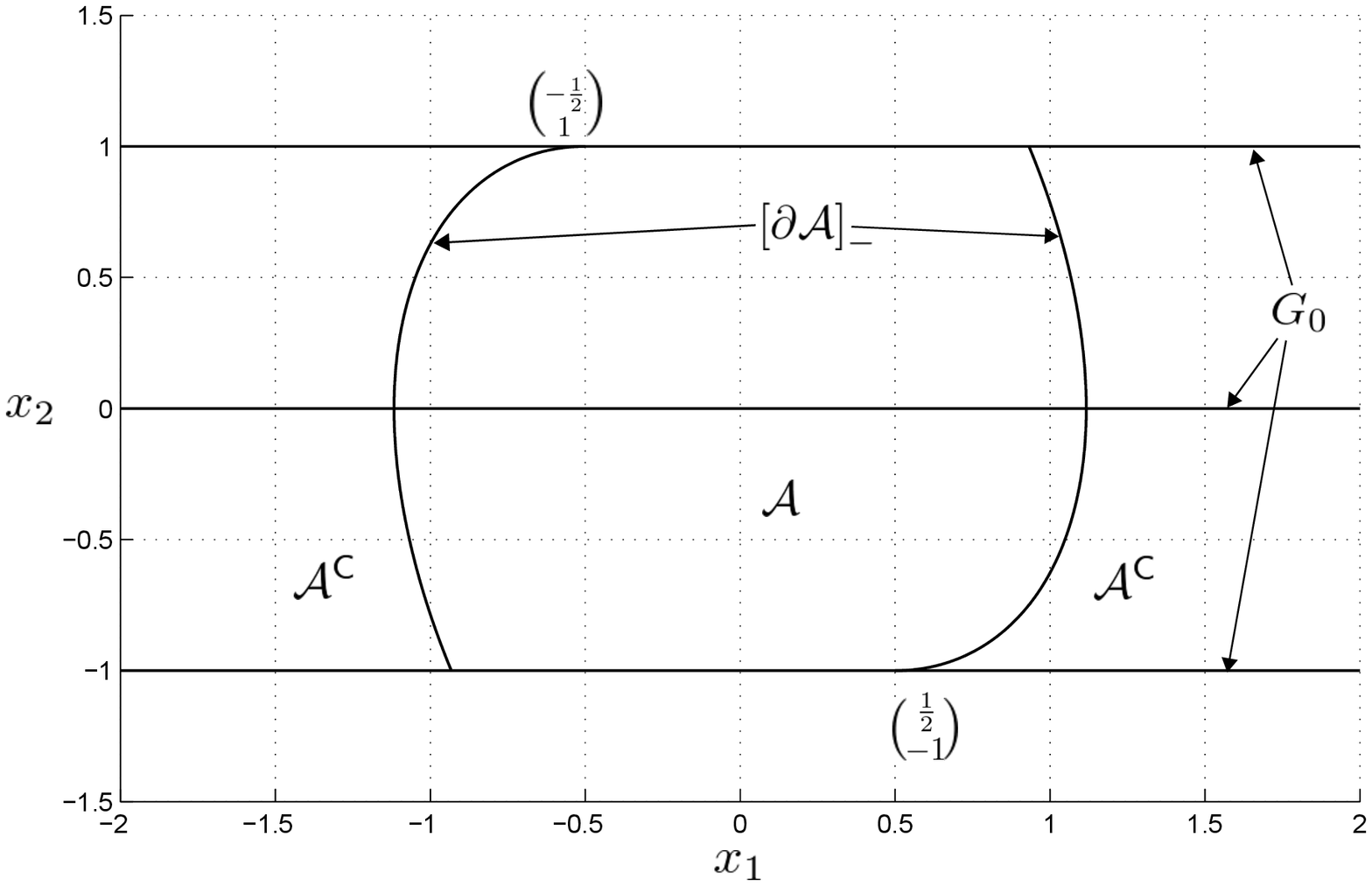}
		\caption{Admissible set of the constrained spring from example \ref{subsec:ConstrainedSpring2}}
	\end{center}
\end{figure}

Let us now turn to the $x_1$ axis, where $\tilde{g} = x_2^2 - |x_2|$ is non differentiable. For any $z$ on the $x_1$ axis, we have $U(z) = [-1,1]$ and $\partial \tilde{g} (z) = \bar{\co}\left( (0,-1)^T, (0,1)^T \right) = \{0\} \times [-1,1]$ and we must have:	
\begin{equation}\label{eq:TanCond1}
	\min_{u\in [-1,1]} \max_{\xi\in\partial \tilde{g}(\tilde{z})} \xi.f(\tilde{z},u) = 0 = \min_{u\in [-1,1]} \max_{\xi_2\in[-1,1]} \xi_2 (-2x_1 + u)
\end{equation}
For each $-\frac{1}{2}\leq z_1\leq \frac{1}{2}$ equation \eqref{eq:TanCond1} has a solution given by $\xi = (0,\sgn(-2z_1 + u))$ from which we deduce that $\bar{u} = 2z_1$. However, one can directly verify that the integral curves of \eqref{diffeq:Ex2} with endpoints in the set $[-\frac{1}{2}, \frac{1}{2}] \times \{ 0 \}$ with the control $u = x_2$ all correspond to admissible curves (integrated backwards) and therefore do not belong to the barrier, but that they make the constraint $g(x^{(\bar{u},\bar{x})}(t),\bar{u}(t))$ equal to 0 for $\bar{u} = x_2$ for all $\bar{x} \in [-\frac{1}{2}, \frac{1}{2}] \times \{ 0 \}$ and for all $t$. This attests that our conditions are only necessary and far from being sufficient.

\begin{rem}
	Note that, as in Example \ref{subsec:ConstrainedSpring1}, Assumption (A4) does not hold true at the final points $z \in G_0$ since there are two active constraints for only one control. Again, we conclude by continuity that condition \eqref{eq:finalConditions} still holds.
\end{rem}

\section{Conclusion}\label{sec:Conclusion}

In this paper we have extended the work on admissible sets and barriers, introduced in \cite{DeDona_siam}, to the case of mixed constraints. In particular, we have shown that the properties of the barrier in the mixed constraint setting prolong those in the pure state constraint setting, with some significant differences concerning its intersection with the set given by $G_0 = \{x : \min_{u\in U} \max_{i = 1,\dots,p} g_i(x,u) = 0 \}$, intersection that occurs tangentially in a generalised sense.

We also had to adapt the minimum-like principle, that allows the barrier's construction, as in Theorem \ref{BarrierTheorem}: a form of the Pontryagin maximum principle, presented in Appendix \ref{AppendixSec:PMPproof}, in terms of the boundary of the reachable set, was needed. However, the result in this form is available only for control functions that are assumed to be piecewise continuous. The possibility of relaxing this assumption to merely measurable controls is an open question, and will be the subject of future works. 

Proving Theorem \ref{extrem:thm} required the introduction of the regularity assumption (A4) to guarantee the existence of needle perturbations that satisfy the constraints, even when some of them are active. This assumption is also used in the proof of the ultimate tangentiality condition \eqref{nonsmoothUltimateTan}. However, assumption (A4) may appear to be too strict, especially on the set $G_0$, since on $G_0$ $\bar{u}$ belongs to the boundary of $U$, thus adding at least one new constraint to the previous ones, and leading to a Jacobian whose lines are no more independent. However, it might be possible to avoid evaluating this rank on $G_0$ by a continuity argument. This point will be addressed in future research.

\appendix

\section{Compactness of solutions}\label{Appendix:Compactness}

We slightly extend the compactness results proven in \cite[Appendix A]{DeDona_siam} to the mixed constraint context. We recall without proof, from \cite{DeDona_siam}, the following lemma and its corollary:
\begin{lem}\label{bound-lem}
If assumptions (A1) and (A2) of Section~\ref{sec:ConsDynCon} hold true, equation (\ref{eq:state_space}) admits a unique absolutely continuous integral curve over $[t_0, +\infty)$ for every $u\in \UU$ and every bounded initial condition $x_{0}$, which remains bounded for all finite $t\geq t_0$,
\begin{equation}\label{bound}
\Vert x(t) \Vert \leq \left( (1+ \Vert x_{0} \Vert^{2} )e^{2 C (t-t_{0})} -1 \right)^{\frac{1}{2}} \triangleq K(t).
\end{equation}

Moreover, we have
\begin{equation}\label{equicont}
\Vert x(t)-x(s) \Vert \leq \bar{C} \vert t-s \vert
\end{equation}
for all $t, s \in [t_0,T]$ and all $T>t_0$, where
\begin{equation}\label{Calphabound}
\bar{C} \triangleq
\sup_{\Vert x \Vert  \leq  K(T), u\in U}  \Vert f(x,u)\Vert < + \infty.
\end{equation}
\end{lem}

\begin{cor}\label{relatcomp-lem}
Let us denote by $\XX(x_{0})$ the set of integral curves issued from an arbitrary $x_{0}$, $\Vert x_{0}\Vert < \infty$, and satisfying (\ref{eq:state_space}), (\ref{eq:initial_condition}), (\ref{eq:input_constraint}).

If assumptions (A1) and (A2) of Section~\ref{sec:ConsDynCon} hold true, $\XX(x_{0})$ is a subset of  $C^{0}([0,\infty), \RR^{n})$, the space of continuous functions from $[0,\infty)$ to $\RR^{n}$, and is relatively compact with respect to the topology of uniform convergence on $C^{0}([0,T], \RR^{n})$ for all finite $T\geq 0$. In other words, from any sequence of integral curves in $\XX(x_{0})$, one can extract a subsequence whose convergence is uniform on every interval $[0,T]$, with $T\geq 0$ and finite, and whose limit belongs to $C^{0}([0,\infty), \RR^{n})$.
\end{cor}
\bigskip

We now adapt the proof of \cite[Lemma A.2, Appendix A]{DeDona_siam}. Since we strictly follow the same lines, only its modifications are presented.
\bigskip

\begin{lem}\label{compact-lem}
Assume that (A1), (A2) and (A3) of Section~\ref{sec:ConsDynCon} hold. Given a compact set $\XX_{0}$ of $\RR^{n}$, the set $\XX\triangleq \bigcup_{x_{0}\in \XX_{0}}\XX(x_{0})$ is compact with respect to the topology of uniform convergence on $C^{0}([0,T], \RR^{n})$ for all $T\geq 0$, namely from every sequence $\{x^{(u_{k},x_{k})}\}_{k\in \NN} \subset \XX$ one can extract a  uniformly convergent subsequence on every finite interval $[0,T]$, whose limit $\xi$ is an absolutely continuous integral curve on $[0,\infty)$, belonging to $\XX$. In other words, there exists $\bar{x}\in \XX_{0}$ and $\bar{u}\in \UU$ such that $\xi(t)= x^{(\bar{u},\bar{x})}(t)$ for almost all $t \geq 0$. 

Moreover, if the sequence $\{(x^{(u_{k},x_{k})}, u_{k})\}_{k\in \NN}$ satisfies the constraint $g(x^{(u_{k},x_{k})}(t), u_{k}(t)) \preceq 0$ for all $k$ and almost all $t$, then the limit also  does: $g(x^{(\bar{u},\bar{x})}(t),\bar{u}(t)) \preceq 0$ for almost all $t$.
\end{lem}

\begin{proof}
Since $\XX_{0}$ is compact, it is immediate to extend inequalities (\ref{bound}) and (\ref{equicont}) to integral curves with arbitrary $x_{0}\in \XX_{0}$ by taking, in the right-hand side of (\ref{bound}), the supremum over all $x_{0}\in \XX_{0}$. Thus,
by the same argument as in the proof of Corollary~\ref{relatcomp-lem}, using Ascoli-Arzel\`{a}'s theorem, we conclude that $\XX$ is relatively compact with respect to the topology of uniform convergence on $C^{0}([0,T], \RR^{n})$, for all $T\geq 0$. The proof that, from every sequence $\{x^{(u_{k},x_{k})}\}_{k\in \NN} \subset \XX$, one can extract a  uniformly convergent subsequence on every finite interval $[0,T]$ whose limit $\xi$ belongs to $\XX$ is done exactly as in \cite{DeDona_siam}.

Accordingly, the sequence of functions $\{ t \mapsto g(x^{(u_{k},x_{k})}(t),u_{k}(t)) : k\in \NN \}$ is bounded in $L^{2}([0,T],\RR^{n})$ for every finite $T$, which implies that this sequence contains at least a weakly convergent subsequence (still denoted by $g(x^{(u_{k},x_{k})},u_{k})$). We denote by $\bar{g}$ its weak limit, independent of $T$ as above.

Recall from \cite{DeDona_siam} that we denote $\bar{x}=\lim_{k\rightarrow \infty} x_{k}$ and $F_{k}(t)= f(x^{(u_{k},x_{k})}(t),u_{k}(t))$.
By Mazur's Theorem (see e.g. \cite[Chapter V, \S1, Theorem 2, p. 120]{yosida}), for every $k$, there exists a sequence $\{\alpha_1^{k}, \ldots, \alpha_{k}^{k}\}$ of non negative real numbers, with $\sum_{i=1}^{k} \alpha_{i}^{k} = 1$, such that the sequence 
$$\left( \begin{array}{c} \tilde{F}_{k}\\ \tilde{g}_{k} \end{array} \right) \triangleq \sum_{i=1}^{k} \alpha_{i}^{k}
\left( \begin{array}{c} F_{i}\\ g(x^{(u_{i},x_{i})},u_{i})\end{array} \right)$$ 
is strongly convergent to $\left( \begin{array}{c} \bar{F}\\ \bar{g} \end{array}\right)$ in every $L^{2}([0,T],\RR^{n})$ for all finite $T$.
Note that this property \emph{a fortiori} holds true if we replace the sequence $F_{i}$ by any subsequence $\left( \begin{array}{c} F_{i}\\ g(x^{(u_{i},x_{i})},u_{i})\end{array} \right)$ constructed by selecting a subsequence of indices $i_{j}$ such that, given $\ee >0$, 
$$\sup_{t\in [0,T]} \left( \Vert f(x^{(u_{i_{j}},x_{i_{j}})}(t), u_{i_{j}}(t)) - f(\xi(t), u_{i_{j}}(t))\Vert + \Vert   g(x^{(u_{i_{j}},x_{i_{j}})}(t), u_{i_{j}}(t)) - g(\xi(t), u_{i_{j}}(t))\Vert \right) < \ee 2^{-j}$$ 
for each $j$, which is indeed possible thanks to the uniform convergence of $x^{(u_{k},x_{k})}$ to $\xi$ and the continuity of $f$ and $g$.  Note also that the limit $\left( \begin{array}{c} \bar{F}\\ \bar{g} \end{array}\right)$ remains the same (for convenience of notation, we keep the same symbols for the $\alpha_{j}^{k}$'s, but we remark that these coefficients have to be adapted relative to the new subsequence).

We therefore deduce, following \cite{DeDona_siam},  that $\bar{F} (t)$ belongs  almost everywhere to the closed convex hull of $\{f(\xi(t),u_{i_{j}}(t))\}_{j\in \NN}$ which is contained in $f(\xi(t), U)$ according to (A3) and, with an obvious adaptation, that $\bar{g}(t) \in g(\xi(t),U)$ for almost all $t$. We immediately conclude that if $g(x^{(u_{k},x_{k})}(t),u_{k}(t)) \preceq 0$ for all $k$ and almost every $t$, it is the same for any convex combination and therefore $\bar{g}(t) \preceq 0$ for almost all $t$.

Finally, again according to (A3) and (A5), there exists, by the measurable selection theorem \cite{Cast_Val}, $\bar{u}\in \UU$ such that 
$$\left( \begin{array}{c} f(\xi(t),\bar{u}(t)) \\ g(\xi(t),\bar{u}(t)) \end{array} \right)
=  \left( \begin{array}{c} \bar{F}(t)\\ \bar{g}(t) \end{array}\right) \quad \mbox{a.e.}~t\in [0,\infty).$$
Thus, we conclude that $\xi$ satisfies $\dot{\xi}=f(\xi,\bar{u})$ almost everywhere, with $\xi(0)=\bar{x}\in \XX_0$. By the uniqueness of integral curves of (\ref{eq:state_space}), we conclude that $\xi(t)=x^{(\bar{u},\bar{x})}(t)$ almost everywhere and, thus, that $\xi\in \XX$. Accordingly, we indeed have $\bar{g}(t)= g(\xi(t),\bar{u}(t)) = g(x^{(\bar{u},\bar{x})}(t),\bar{u}(t)) \preceq 0$ a.e. $t\in [0,\infty)$, which achieves to prove the lemma.
\end{proof}

\section{Maximum principle for problems with mixed constraints}\label{Appendix:PMP}
~

In this appendix we sketch a version of the maximum principle for problems with mixed constraints describing the extremal curves as those whose endpoints at each time $t$ belong to the boundary of the reachable set at the same instant of time. This form is useful to prove Theorem \ref{BarrierTheorem}. The proof draws content from \cite{Lee_Markus}, where the principle is proved in the particular case of constraints on the control, and \cite{PBGM}, where the principle is proved in the context of optimising a cost function for systems with both constraints on the control and the state, but which are not mixed, though a remark indicating the possibility of its extension to mixed constraints is given in \cite[Chapter VI, \S 35]{PBGM}. See also \cite[Chapter 7]{Hestenes} for a proof in the framework of the Calculus of Variations.  For a survey on the maximum principles with mixed constraints, the reader may refer to \cite{Hartl_et_al}.

In our treatment we will introduce the suitable perturbations to \emph{regular trajectories}, similar to \cite{PBGM}, that are needed to generate the so-called perturbation cone, the latter being crucial to obtain the necessary conditions of the maximum principle. Throughout the analysis we assume that the extremal control is piecewise continuous as in the above cited references.

\subsection{Control perturbations}\label{AppendixSec:ElementaryPert}
~

Consider an integral curve $x^{(\bar{u},x_0)}$ associated with the piecewise continuous control $\bar{u}$, initiating from the point $x_0$. Let $\tau_k$, $k = 0,\dots,K$, with $\tau_0 = 0$, be a collection of points of continuity of $\bar{u}$ such that $\tau_k - \ee l_k$ is also a point of continuity with $l_k \geq 0$ for all $\ee$ small enough. Assume that $g(x^{(\bar{u},x_0)}(t),\bar{u}(t)) \preceq 0$ for \emph{a.e.} $t \in [\tau_{k-1}, \tau_{k}[$. We will perturb the control over the interval $I_k = [\tau_k - \ee l_k,\tau_k[$ and extend both the control and the integral curve between $\tau_k$ and $\tau_{k+1} - \ee l_{k+1}$ in order to satisfy the constraints. This will be done by first making a subdivision $\sigma_{q}^k$, $q = 1,\dots,d_k$, namely $\tau_k = \sigma_1^k < \dots <\sigma_q^k < \dots < \sigma_{d_k}^k = \tau_{k + 1} - \ee l_{k + 1}$, assumed to contain all discontinuities of $\bar{u}$ on the interval $[\tau_{k}, \tau_{k+1} - \ee l_{k + 1}[$ and adapting $\bar{u}$ and its corresponding integral curve on each subinterval $[\sigma_q^k, \sigma_{q+1}^k[$, $q = 1,\dots,d_k - 1$, using the implicit function theorem.

At $k = 1$, if $s_1 = \#\II(x^{(\bar{u},x_0)}(\tau_1), \bar{u}(\tau_1)) = 0$ and $s_2 = \#\JJ(\bar{u}(\tau_1)) = 0$ (we do not pass on the indexing of $s_1$ and $s_2$ with respect to $k$ and $q$ to avoid too cumbersome notations), we introduce the classical needle perturbation: $\bar{u}_{I_1} \triangleq \bar{u} \Join_{\tau_{1} - \ee l_{1}} v_{I_1}$, defined on the interval $[0,\tau_1[$, with $v_{I_1}$ arbitrarily chosen in $U(x^{(\bar{u},x_0)}(\tau_1 - \ee l_1))$ and constant over $[\tau_1 - \ee l_1, \tau_1[$. We denote by $\xi_{I_1} = x^{(\bar{u}_{I_1},x_0)}(\tau_1)$.

Otherwise, by remarking that $s_1$ and $s_2$ are such that $\max(s_1, s_2) > 0$, according to (A4), define the function
\[
\Gamma_{I_1}(x,u) = \left(\begin{array}{c}
g_{i_1}(x,u)\\
\dots \\
g_{i_{s_1}}(x,u)\\
\gamma_{j_1}(u)\\
\dots\\
\gamma_{j_{s_2}}(u)
\end{array}\right)
\]
and consider the solution $\hat{u}_{I_1}$ of $\Gamma_{I_1}(x,u) = 0$, defined from a neighbourhood $\Nnn_{I_1}$ of $\RR^n \times \RR^{m - (s_1 + s_2)}$ to $\RR^{s_1 + s_2}$. Thus,
$$\Gamma_{I_1}(x,\hat{u}_{I_1}(x,u_{s_1 + s_2 + 1},\dots,u_{m}),u_{s_1 + s_2 + 1},\dots,u_{m}) = 0 \quad \forall (x,u_{s_1 + s_2 +1},\dots,u_{m}) \in \Nnn_{I_1}$$ and we can define the integral curve $x_{I_1}$ by:
\begin{equation}
\dot{x} = f(x,\hat{u}_{I_1}(x,v_{s_1 + s_2 + 1},\dots,v_m),v_{s_1 + s_2 + 1},\dots,v_m)
\end{equation}
starting from $x^{(\bar{u},x_0)}(\tau_1 - \ee l_1)$, with $(v_{s_1 + s_2 + 1}, \dots, v_m)$ arbitrary in the projection of $\Nnn_{I_1}$ on $\RR^{m - (s_1 + s_2)}$, and such that $\bar{u}_{I_1}(t) \triangleq \hat{u}_{I_1}(x_{I_1}(t),v_{s_1 + s_2 + 1},\dots,v_m) \in U(x_{I_1}(t))$ for all $t\in [\tau_1 - \ee l_1, \tau_1[$. In this case we denote $\xi_{I_1} = x_{I_1}(\tau_1)$. 

We now consider the interval $[\tau_1, \tau_2 - \ee l_2]$. If $s_1 = \#\II(x^{(\bar{u},x_0)}(\tau_1), \bar{u}(\tau_1)) = 0$ and $s_2 = \#\JJ(\bar{u}(\tau_1)) = 0$ then $\bar{u}$ is kept the same on $[\sigma_1^1, \sigma_2^1[$ and we denote by $\bar{u}_{I_{1,1}} = \bar{u}_{I_1} \Join_{\tau_1} \bar{u}$ and $\xi_{I_{1,1}} = x^{(\bar{u},\xi_{I_1})}(\sigma^1_2 - \sigma^1_1)$. Otherwise, since $s_1 = \#\II(x^{(\bar{u},x_0)}(\tau_1), \bar{u}(\tau_1))$ and $s_2 = \#\JJ(\bar{u}(\tau_1))$ are such that $\max(s_1, s_2) > 0$, according to (A4), define the function
\[
\Gamma_{I_{1,1}}(x,u) = \left(\begin{array}{c}
g_{i_1}(x,u)\\
\dots \\
g_{i_{s_1}}(x,u)\\
\gamma_{j_1}(u)\\
\dots\\
\gamma_{j_{s_2}}(u)
\end{array}\right)
\]
and consider the solution $\hat{u}_{I_{1,1}}$ of $\Gamma_{I_{1,1}}(x,u) = 0$, defined from a neighbourhood $\Nnn_{I_{1,1}}$ of $\RR^n \times \RR^{m - (s_1 + s_2)}$ to $\RR^{s_1 + s_2}$. Thus,
$$\Gamma_{I_{1,1}}(x,\hat{u}_{I_{1,1}}(x,u_{s_1 + s_2 + 1},\dots,u_{m}),u_{s_1 + s_2 + 1},\dots,u_{m}) = 0 \quad \forall (x,u_{s_1 + s_2 +1},\dots,u_{m}) \in \Nnn_{I_{1,1}}$$ and we can define the integral curve $x_{I_{1,1}}$ by:
\begin{equation}
\dot{x} = f(x,\hat{u}_{I_{1,1}}(x,\bar{u}_{s_1 + s_2 + 1},\dots,\bar{u}_m),\bar{u}_{s_1 + s_2 + 1},\dots,\bar{u}_m)
\end{equation}
starting from $\xi_{I_1}$ at time $\tau_1$ and assume that the interval $[\sigma_1^1, \sigma_2^1[$ is small enough such that its solution remains in $\Nnn_{I_{1,1}}$.

We iteratively apply the same construction for all $q = 2,\dots, d_1$ and thus obtain the perturbed $x_{I_{1,q}}$ and $\bar{u}_{I_{1,q}}$ in each interval $[\sigma^1_q,\sigma^1_{q+1}[$, $q = 1,\dots, d_k - 1$, satisfying the constraints.

Then finally, for $k > 1$, assuming that $x_{I_{k,d_{k}-1}}$ and $\bar{u}_{I_{k,d_{k}-1}}$ have been obtained, we construct $x_{I_{k+1,d_{k+1} - 1}}$ and $\bar{u}_{I_{k+1,d_{k + 1} - 1}}$ by replacing $\tau_1$ in the above algorithm by $\tau_k$ to finally get the complete perturbed trajectory.

According to \cite[Chapter VI, \S 34]{PBGM} we introduce the following notations: the perturbation parameters denoted by $\pi$ belong to the convex cone $\co \{ (\tau_k,l_k,v_k,\ee) : k = 1,\dots,K \}$ and we note $x_{\pi}(t) = x_{I_{k,q}}(t)$ previously defined with the vector of perturbation parameters $\pi$ if $t\in [\sigma_q^k,\sigma_{q+1}^k[$. Then, for a given vector of perturbation parameters $\pi \triangleq \{\tau_1,\dots,\tau_k,\alpha_1l_1,\dots,\alpha_kl_k,v_1,\dots,v_k\}$, with $\alpha_k \geq 0$ and $\sum_{k=1}^K \alpha_k = 1$, we have

\begin{equation}\label{perturbationResult}
x_{\pi}(t) = x^{(\bar{u},x_0)}(t) + \ee \delta x(t) + O(\ee^2)
\end{equation}
with 
\begin{equation}\label{needlePerturbation:deltaX}
\delta x(t) = \sum_{k = 1}^{K} \alpha_k \Phi^{\bar{u}}(t,\tau_{k}) \left[f(x^{(\bar{u},x_0)}(\tau_k),v_k) - f(x^{(\bar{u},x_0)}(\tau_k), \bar{u}(\tau_k)) \right]l_k
\end{equation}
and $\Phi^{\bar{u}}$ the transition matrix of the variational equation:
\begin{equation}\label{eq:VariationalEquation}
\frac{d}{dt}(\Phi^{\bar{u}}(t,\tau)) = \left( \frac{\partial f}{\partial x}(x^{(\bar{u},x_0)}(t),\bar{u}(t)) + \Lambda^{\bar{u}}(t)\frac{\partial g}{\partial x}(x^{(\bar{u},x_0)}(t),\bar{u}(t))\right) \Phi^{\bar{u}}(t,\tau),\quad \Phi^{\bar{u}}(\tau,\tau) = I
\end{equation}
for all $0\leq \tau \leq t$, $\Lambda^{\bar{u}}$ being the piecewise continuous solution of the equation
\begin{equation}
\frac{\partial f}{\partial u}(x^{(\bar{u},x_0)}(t),\bar{u}(t)) + \Lambda^{\bar{u}}(t)\frac{\partial g}{\partial u}(x^{(\bar{u},x_0)}(t),\bar{u}(t)) = 0.
\end{equation}

Now introducing $\eta^{\bar{u}}(t) = (\Phi^{\bar{u}})^{-1}(t,0)\eta_0 = \Phi^{\bar{u}}(0,t)\eta_0$ for an arbitrary $\eta_0\neq 0$ 
and setting 
\begin{equation}\label{def:mu}
\mu^{\bar{u}}(t) \triangleq - \Lambda^{\bar{u}}(t)^T\eta^{\bar{u}}(t)
\end{equation} we get the adjoint equation
\begin{equation}\label{adjoint-eta}
\dot{\eta}^{\bar{u}}(t)= - \left( \frac{\partial f}{\partial x}(x^{(\bar{u},x_{0})}(t),\bar{u}(t))\right)^T\eta^{\bar{u}}(t) + \sum_{i=1}^p \mu^{\bar{u}}_i(t)\frac{\partial g_i}{\partial x}^T(x^{(\bar{u},x_{0})}(t),\bar{u}(t))
\end{equation}
and it can be proven that
\begin{equation}\label{VarAdjConstant}
\frac{d}{dt}(\delta x(t)^T \eta^{\bar{u}}(t)) = 0\quad \forall t \quad \forall \eta_0 \neq 0
\end{equation}

According to \eqref{needlePerturbation:deltaX}, for all perturbation parameters $\pi$, we have thus defined the so-called \emph{tangent perturbation cone}, denoted by $\mathcal{K}_t$ and $\eta^{\bar{u}}$ may be interpreted as the normal to the separating hyperplane to $\mathcal{K}_t$; moreover $\mu^{\bar{u}}$ defined by \eqref{def:mu} may be interpreted as the Karush-Kuhn-Tucker multiplier associated with the constraints $g \preceq 0$. The interested reader may refer to \cite{PBGM} or \cite{Hestenes}.

\subsection{The maximum principle}\label{AppendixSec:PMPproof}
~

The following theorem is an adaptation of \cite[Theorem 23, Chapter VI, \S 35]{PBGM}, in the spirit of  \cite{Lee_Markus}, using the perturbation cone constructed in the previous section.
\begin{thm}[Maximum principle]\label{extrem:thm}
Consider the constrained system (\ref{eq:state_space}), (\ref{eq:initial_condition}), (\ref{eq:input_constraint}) (\ref{eq:state_const}). Let $x^{(\bar{u},x_{0})}$ be a regular trajectory associated with the piecewise continuous control $\bar{u}\in \UU$ such that $x^{(\bar{u},x_{0})}(t_{1}) \in \partial R_{t_{1}}(x_{0})$ for some $t_{1}>0$. Then, there exists a non zero absolutely continuous $\eta^{\bar{u}}$ and piecewise continuous multipliers $\mu^{\bar{u}}_i\geq 0$, $i = 1,\dots,p$ satisfying,  for almost all $t \leq t_1$:
\begin{equation}
\dot{\eta}^{\bar{u}}(t)= - \left( \frac{\partial f}{\partial x}(x^{(\bar{u},x_{0})}(t),\bar{u}(t))\right)^T\eta^{\bar{u}}(t) + \sum_{i=1}^p \mu^{\bar{u}}_i(t)\frac{\partial g_i}{\partial x}^T(x^{(\bar{u},x_{0})}(t),\bar{u}(t))
\end{equation}
\begin{equation}\label{eq:CompSlackCond}
\mu^{\bar{u}}_i(t)g_i(x^{(\bar{u},x_{0})}(t),\bar{u}(t)) = 0 \quad \forall i\in\{1,\dots,p\}
\end{equation}
such that, if we define the dualised Hamiltonian
\begin{equation}\label{eq:HamiltonianDef}
\mathcal{H}(x,u,\eta,\mu) \triangleq \eta^Tf(x,u) + \sum_{i = 1}^p \mu_i g_i(x,u),
\end{equation}
it satisfies
\begin{equation}\label{barriercond-eta} 
\max_{u\in U} \mathcal{H}(x^{(\bar{u},x_{0})}(t),u,\eta^{\bar{u}}(t),\mu^{\bar{u}}(t)) = \mathcal{H}(x^{(\bar{u},x_{0})}(t),\bar{u}(t),\eta^{\bar{u}}(t),\mu^{\bar{u}}(t)) = constant \quad \mathrm{a.e.}~ t \leq t_1
\end{equation}

\end{thm}

\bibliographystyle{siam}
\bibliography{Hybrid_bib}

\begin{thebibliography}{10}

\bibitem{berge}
{\sc C.~Berge}, {\em Topological Spaces}, Oliver and Boyd, Edinburgh and
  London, 1963.

\bibitem{Cast_Val}
{\sc C.~Castaing and M.~Valadier}, {\em Convex Analysis and Measurable
  Multifunctions}, vol.~580 of Lecture Notes in Mathematics, Springer, 1977.

\bibitem{Clarke}
{\sc F.H. Clarke}, {\em Optimization and Nonsmooth Analysis}, John Wiley \&
  Sons, Inc., New York, 1983.

\bibitem{Clarke_DiPinho}
{\sc F.H. Clarke and M.~de~Pinho}, {\em Optimal control problems with mixed
  constraints}, SIAM J Control Optim., 48 (2010), pp.~4500--4524.

\bibitem{Clarke_et_al_springer}
{\sc F.H. Clarke, Yu.S. Ledyaev, R.J. Stern, and P.R. Wolenski}, {\em Nonsmooth
  Analysis and Control Theory}, Springer-Veriag, New York, 1998.

\bibitem{danskin}
{\sc J.~Danskin}, {\em The Theory of Max-Min}, Springer, 1967.

\bibitem{DeDona_siam}
{\sc J.A. De~Dona and J.~L{\'e}vine}, {\em On barriers in state and input
  constrained nonlinear systems}, SIAM J. Control Optim., 51 (2013),
  pp.~3208--3234.

\bibitem{Hartl_et_al}
{\sc R.F. Hartl, S.P. Sethi, and R.J. Vickson}, {\em A survey of the maximal
  principles for optimal control problems with state constraints}, SIAM Review,
  37 (1995), pp.~181--218.

\bibitem{heinonen}
{\sc J.~Heinonen}, {\em Lectures on {L}ipschitz analysis}, in Lectures at the
  14th Jyv{\"a}skyl{\"a} Summer School, August 2004.

\bibitem{Hestenes}
{\sc M.~R. Hestenes}, {\em Calculus of Variations and Optimal Control Theory},
  John Wiley, 1966.

\bibitem{Isaacs}
{\sc R.~Isaacs}, {\em Differential Games}, John Wiley {\&} Sons, Inc., 1965.

\bibitem{Lee_Markus}
{\sc E.~B. Lee and L.~Markus}, {\em Foundations of Optimal Control Theory}, The
  SIAM Series in Applied Mathematics, John Wiley \& Sons, Inc., New York, 1967.

\bibitem{JLbook}
{\sc J.~L{\'e}vine}, {\em Analysis and Control of Nonlinear Systems: A
  Flatness-Based Approach}, Mathematical Engineering, Springer, 2009.

\bibitem{NicotraNalGor_IFACE2014}
{\sc M.~Nicotra, R.~Naldi, and E.~Garone}, {\em Taut cable control of a
  tethered uav}, in Proceedings of the 19th International Federation of
  Automatic Control World Congress, vol.~19, 2014, pp.~3190--3195.

\bibitem{Pesch94apractical}
{\sc H.~J. Pesch}, {\em A practical guide to the solution of real-life optimal
  control problems}, Control and Cybernetics, 23 (1994).

\bibitem{PBGM}
{\sc L.~Pontryagin, V.~Boltyanskii, R.~Gamkrelidze, and E.~Mishchenko}, {\em
  The Mathematical Theory of Optimal Processes}, John Wiley {\&} Sons, Inc.,
  1965.

\bibitem{Sira}
{\sc H.~Sira-Ramirez and S.K. Agrawal}, {\em Differentially flat systems},
  Marcel Dekker, Inc., 2004.

\bibitem{yosida}
{\sc K.~Yosida}, {\em Functional Analysis}, {S}pringer-{V}erlag, 1971.

\end{thebibliography}

\end{document}